\documentclass[11pt, a4paper]{article}

\usepackage[english]{babel}

\usepackage[a4paper,top=2cm,bottom=2cm,left=3cm,right=3cm,marginparwidth=1.75cm]{geometry}

\usepackage{enumerate}

\usepackage[shortlabels]{enumitem}

\usepackage{amsmath}
\usepackage{amsfonts}
\usepackage{graphicx}
\usepackage[colorlinks=true, allcolors=black]{hyperref}
\usepackage{parskip}
\usepackage{tikz}
\usepackage{tikz-cd}
\usetikzlibrary{decorations.markings}
\usepackage{amsthm}
\usepackage{amssymb}
\usepackage{mathtools}

\usetikzlibrary{lindenmayersystems,arrows.meta}

\newcommand{\Z}{\mathbb{Z}}

\newcommand{\AwrB}{{A \operatorname{wr} B}}

\newtheorem{theorem}{Theorem}[subsection]
\newtheorem*{theorem*}{Theorem}
\newtheorem{cor}[theorem]{Corollary}

\newtheorem*{claim*}{Claim}
\newtheorem{lem}[theorem]{Lemma}
\newtheorem{prop}[theorem]{Proposition}

\theoremstyle{definition}
\newtheorem{definition}[theorem]{Definition}

\newtheorem{rem}[theorem]{Remark}

\newtheorem*{rem*}{Remark}

\usepackage{lipsum}
\usepackage[font=itshape]{quoting}

\title{The Conjugacy Problem in Wreath Products}
\author{Sara Luder}

\date{}

\begin{document}

\maketitle

\begin{abstract}

In 1966 Jane Matthews claimed that the conjugacy problem is solvable in the standard restricted wreath product $A \wr B$ of two nontrivial groups $A$ and $B$ if and only if (i) the conjugacy problem is solvable in $A$ and $B$ and (ii) $B$ has a \textit{solvable power problem}.  We show that there should be an additional condition that either $A$ is abelian or $B$ has a \textit{solvable order problem}. We also show that, if $A$ and $B$ are non-trivial recursively presented groups where $A$ has an infinite number of conjugacy classes and $B$ acts on $B/H$ transitively, then the conjugacy problem in the permutational restricted wreath product $A \wr_{B/H} B$ is solvable if and only if the following hold: (1) the conjugacy problem is solvable in $A$ and in $B$; (2) either $A$ is abelian or \textit{the orbit order problem} is solvable in $B$; (3) for any $\gamma, \beta \in B$ the membership problem for $H \gamma \langle \beta \rangle$  is solvable; and (4) for any $\beta \in B$ and any finite set of $n$ pairs of elements $(\alpha_i, \gamma_i)$, where $\alpha_i, \gamma_i \in B$ we can determine whether or not $ \big{[} \bigcap_{i=1}^n \alpha^{-1}_iH\gamma_i \langle \beta \rangle\big{]} \cap C_B(\beta) = \emptyset $.   
\end{abstract}
  
\section{Introduction}

In 1911 Max Dehn proposed three fundamental algorithmic problems in group theory: the word problem, the conjugacy problem and the group isomorphism problem. Solving the word problem requires the existence of an algorithm to determine whether any word $w$ on the generators of a group $G$ is the identity.  The conjugacy problem requires the existence of an algorithm to determine whether two words $w_1$ and $w_2$ are conjugate in a group $G$, which we denote $w_1 \sim_G w_2$. If the conjugacy problem is solvable, then the word problem is also solvable (by setting $w_2 =1$) but the converse is not true. (See for example Theorem 9 page 31 \cite{M3} where Miller constructs a finitely presented group with a solvable word problem but an unsolvable conjugacy problem.) 

A Group $G$ is \textit{recursively presented} if $G = \langle g_1 \ldots g_m \: | \: w=1 \: \forall w \in E \rangle$ where $m$ is finite, each $w$ is a word in $G$ and $E$ is a recursively enumerated set. A group is \textit{Class $\mathcal{C}$} if it is recursively presented and has a solvable conjugacy problem for a given presentation.  
 
In Theorem B \cite{M1}, Jane Matthews sought to prove that the standard restricted wreath product $A \wr B$ of two nontrivial groups $A$ and $B$ both of Class $\mathcal{C}$ is itself of Class $\mathcal{C}$ if and only if the group $B$ has a \textit{solvable power problem}.

Matthews defined the \textit{solvable power problem} as the ability to decide for any two elements $x$ and $y$ in $B$ whether there exists an integer $n$ such that $y = x^n$.  We discuss in section \ref{OO} below why this condition is insufficient if $A$ is not abelian, and we restate Theorem B \cite{M1} with the addition of (ii) as follows:

\bigskip

\begin{theorem} \label{B}

The standard restricted wreath product $A \wr B$ of two nontrivial groups $A$ and $B$ both of Class $\mathcal{C}$ is itself of Class $\mathcal{C}$ if and only if (i) $B$ has a solvable power problem and (ii) either $A$ is abelian or $B$ has a solvable order problem.

\end{theorem}

We define the \textit{order problem} in a group $B$ as the ability to decide for any $x \in B$ whether or not the order of $x$ in $B$ is finite.  We note that, if we know that $x$ has finite order, it is a straightforward (if potentially time-consuming) application of the word problem to find the order of $b$.

We next consider whether Theorem \ref{B} above can be extended to the conjugacy problem for the permutational restricted wreath product $A \wr_S B$ where, rather than acting on itself, the group $B$ acts on a set $S$. 

The action of a group $B$ on a set $S$ is \textit{transitive} if for any $s_o, s \in S$, $s = s_0\beta$ for some $\beta \in B$.  If the action of $B$ on $S$ is transitive, then $S = B/H$ for some $H \leq B$.

We can solve the \textit{membership problem} for $K \gamma L$, where $K,L \leq G$ and $\gamma \in G$ if for any $g \in G$ we can determine whether $g \in K \gamma L$. 

In Lemma \ref{pres2}, we restate a result of Yves Cornulier in \cite{C3} which shows that, if (i) both $A$ and $B$ are Class $\mathcal{C}$, (ii) the action of $B$ on $S$ is transitive, and (iii) we can solve the membership problem for $H$ then $A \wr_S B$ may be recursively presented.

Our main results are as follows:

\bigskip

\begin{theorem} \label{oneway}

Let $A$ and $B$ be non-trivial recursively presented groups, where $B$ acts on a set $S$ transitively, so $S = B/H$ where $H$ is the  stabiliser of some arbitrary point $s_0 \in S$.  The conjugacy problem in the permutational restricted wreath product $A \wr_S B$ is solvable if:

\begin{enumerate}[(i)]
\item the conjugacy problem is solvable in $A$ and in $B$;
\item either $A$ is abelian or the orbit order problem is solvable in $B$;
\item the membership problem for $H \gamma \langle \beta \rangle$ (for any $\gamma, \beta \in B$) is solvable; and
\item for any $\beta \in B$ and any finite set of $n$ pairs of elements $(\alpha_i, \gamma_i)$, where $\alpha_i, \gamma_i \in B$ and $1 \leq i \leq n$, we can determine whether or not $\big{[}\bigcap_{i=1}^n \alpha^{-1}_iH\gamma_i \langle \beta \rangle \big{]}  \cap C_B(\beta) = \emptyset $.    \end{enumerate}  
  
\end{theorem}

The \textit{orbit order problem} is the ability to decide for any $b \in B$ whether or not the number of elements of $S$ in that $b$-orbit is finite. This is equivalent to being able to determine for any $b \in B$ whether there are a finite or infinite number of $H$ cosets in $H\langle b \rangle$.  We discuss this further in Section \ref{OO} below. 

By making the additional assumption that $A$ has an infinite number of conjugacy classes, we may restate this as if and only if:  

\bigskip

\begin{theorem} \label{conditions}

Let $A$ and $B$ are non-trivial recursively presented groups, where $A$ has an infinite number of conjugacy classes and $B$ acts on a set $S$ transitively, so $S = B/H$ where $H$ is the  stabiliser of some arbitrary point $s_0 \in S$.  The conjugacy problem in the permutational restricted wreath product $A \wr_S B$ is solvable if and only if:

\begin{enumerate}[(i)]
\item the conjugacy problem is solvable in $A$ and in $B$;
\item either $A$ is abelian or the orbit order problem is solvable in $B$;
\item the membership problem for $H \gamma \langle \beta \rangle$ (for any $\gamma, \beta \in B$) is solvable; and
\item for any $\beta \in B$ and any finite set of $n$ pairs of elements $(\alpha_i, \gamma_i)$, where $\alpha_i, \gamma_i \in B$ and $1 \leq i \leq n$, we can determine whether or not $\big{[}\bigcap_{i=1}^n \alpha^{-1}_iH\gamma_i \langle \beta \rangle \big{]}  \cap C_B(\beta) = \emptyset $.    \end{enumerate}  
  
\end{theorem}

We also consider the assumption that the action of $B$ on $S$ is transitive, and show the following: 

\bigskip

\begin{theorem} 

Let $A$ and $B$ be non-trivial recursively presented groups, where $B$ acts on a set $S$ such that there are a finite number $m$ of $B$-orbits in $S$, so:

\[S =\sum_{k=1}^m  S_k =  \sum_{k=1}^m  B/H_k \]

where each $S_k$ is a $B$-orbit in $S$ and each $H_k$ is the stabiliser of some arbitrary point in $S_k$.  The conjugacy problem in the permutational restricted wreath product $A \wr_S B$ is solvable if:

\begin{enumerate}[(i)]
\item the conjugacy problem is solvable in $A$ and in $B$;
\item either $A$ is abelian or the orbit order problem is solvable in $B$;
\item for each $H_k$ the membership problem for $H_k \gamma \langle \beta \rangle$ (for any $\gamma, \beta \in B$) is solvable;
\item if for each $S_k$, for any $\beta \in B$ and any finite set of $n$ pairs of elements $(\alpha_i, \gamma_i)$, where $\alpha_i, \gamma_i \in B$ and $1 \leq i \leq n$: 

\[\Gamma_k \coloneq  \big{[}\bigcap_{i=1}^n \alpha^{-1}_iH_k\gamma_i \langle \beta \rangle \big{]}  \cap C_B(\beta) \]

we can determine whether $ \bigcap_{k=1}^m \Gamma_{k} =\emptyset$. 
\end{enumerate}  
  
\end{theorem}

Our approach initially follows that taken by Matthews in \cite{M1} for $A \wr B$, by taking an arbitrary element $bf$ in the permutational restricted wreath product $ A \wr_S B$, dividing the support of $f$ into its $b$-orbits and mapping each orbit to a single element of $A$.  In Section \ref{perm} we use this technique to prove Theorem \ref{conditions}: 

We then apply Theorem \ref{conditions} to prove Theorem \ref{B} and consider some other applications of Theorems \ref{oneway} and \ref{conditions}.

The first is a direct result of Theorem \ref{conditions}

\bigskip

\begin{cor}

If the membership problem is not solvable for $H$, then the conjugacy problem is not solvable in $A \wr_{B/H} B$.

\end{cor}

We can use Theorem \ref{conditions} to construct examples of permutational wreath products with an unsolvable conjugacy problem.  

For example, in \cite{M6}, Mikhailova showed that $F_2 \times F_2$ contains a subgroup $M(G)$ which is finitely generated but has an unsolvable membership problem.  Here $M(G)$ is the fibre product $\{(u,v) \in F_2 \times F_2 \: | \: \Pi(u) = \Pi(v)\} $ where $G$ is a quotient of $F_2$, the word problem is not solvable in $G$ and there is a surjective homomorphism $\Pi: F_2 \twoheadrightarrow G$.  Hence by Theorem \ref{conditions}, if $B = F_2 \times F_2$ and $H = M(G)$, the conjugacy problem is not solvable in $A \wr_{B/H} B$.

We can also use Theorems \ref{oneway} and \ref{conditions} to test whether the conjugacy problem is solvable in certain wreath products.  In particular, we prove the following:

\bigskip

\begin{cor}
    
Let $A$ and $B$ are non-trivial recursively presented groups, where $A$ has an infinite number of conjugacy classes and $B$ is finite and acts on a set $S$ transitively, so $S = B/H$ where $H$ is the stabiliser of some arbitrary point $s_0 \in S$.  The conjugacy problem in the permutational restricted wreath product $A \wr_S B$ is solvable if and only if the conjugacy problem is solvable in $A$.

\end{cor}

\bigskip

\begin{cor} \label{polythe}

If $A$ is a recursively presented group in which the conjugacy problem is solvable and $B$ is a polycyclic group that acts on a set $S$ transitively, then the conjugacy problem is solvable in $A \wr_S B$.

\end{cor}

In particular, as all finitely generated abelian groups are polycyclic:

\bigskip

\begin{cor}

If $A$ is a recursively presented group in which the conjugacy problem is solvable and $B$ is a finitely generated abelian group that acts on a set $S$ transitively, then the conjugacy problem is solvable in $A \wr_S B$.  

\end{cor}

Lastly:

\bigskip

\begin{cor} \label{hypthe}

If $A$ is a recursively presented group in which the conjugacy problem is solvable and $B$ is a hyperbolic group that acts on a set $S$ transitively, with $S = B/H$ and $H \leq B$ is quasi-convex then the conjugacy problem is solvable in $A \wr_S B$.

\end{cor}

\section{Preliminaries}

We start with some terminology.  Throughout this paper $A$ and $B$ are non-trivial groups and we use $1_A$ and $1_B$ for their respective identities.

\bigskip

\begin{definition}

Let $A$ and $B$ be groups and let $B$ act on a set $S$.  Then:

\[ A^S = \{f: S \rightarrow A\}\]

where, if $f, g \in A^S$ then $fg (s) = f(s)g(s)$ for all $s \in S$.  

\end{definition}

We use $1 = (1_A, 1_A \ldots)$ for the identity in $A^S$.

\bigskip

\begin{definition}

Let $f \in A^S$.  The support of $f$, $\sigma(f) = \{s \in S \: | \: f(s) \neq 1_A \}$
    
\end{definition}

\bigskip

\begin{definition}

Let $A$ be a group and $S$ be a set.  Then $A^{(S)} = \{f: S \rightarrow A\}$ with the additional condition that $\sigma(f)$ is finite.  So $A^{(S)} \subseteq A^S$.

\end{definition}

\bigskip

\begin{definition}

$f^b(s) = f(sb^{-1})$ for all $s \in S$, $ f \in A^S$ and $b \in B$. 

\end{definition}

We note that we use $f(sb^{-1})$, rather than $f(sb)$, because this is the element of $f$ that is brought to position $s$ by $b$.  In other words, as this is a right action, the action is associative with:

\[(f^b)^c(s) = f(sc^{-1}b^{-1}) = f^{bc}(s).\]

We now define the various wreath products and specify the terminology that we will use for each.

In the standard wreath product $B$ acts on itself by right multiplication.

\bigskip

\begin{definition} 

The \textit{standard unrestricted wreath product} is: 
\[\AwrB = A^B \rtimes B = \{ (b,f) \:| \: b \in B, f \in A^B \}\]
where $B$ acts on itself by right multiplication and $(b,f)(c,g) = (bc, f^cg)$ for all $b,c \in B$ and $f,g \in A^B$.

\end{definition}

\bigskip

\begin{definition} 

 The \textit{standard restricted wreath product} is: 
 
 \[A \wr B = A^{(B)} \rtimes B = \{ (b,f) \: | \: b \in B, f \in A^{(B)} \}\]
 
where $|\sigma(f)|$ is finite, $B$ acts on itself by right multiplication and $(b,f)(c,g) = (bc, f^cg)$ for all $b,c \in B$ and $f,g \in A^{(B)}$.

\end{definition}

\bigskip

\begin{definition} 

The \textit{permutational unrestricted wreath product} is: 

\[A \operatorname{wr}_S B  =A^S \rtimes B = \{ (b,f) \:| \: b \in B, f \in A^S \}\]

where $B$ acts on a set $S$ and $(b,f)(c,g) = (bc, f^cg)$ for all $b,c \in B$ and $f,g \in A^S$.

\end{definition}

\bigskip

\begin{definition} 

The \textit{permutational restricted wreath product}:

\[A \wr_S B = A^{(S)} \rtimes B =  \{ (b,f) \: | \: b \in B, f \in A^{(S)} \}  \]

where $|\sigma(f)|$ is finite, $B$ acts on $S$ and $(b,f)(c,g) = (bc, f^cg)$ for all $b,c \in B$ and $f,g \in A^{(S)}$.
\end{definition}

When $B$ acts on a set $S$, let $H$ be the stabiliser of $s_0$, an arbitrary element of $S$, so:

\[H = \hbox{Stab}_B(s_0) = \{ b \in B \:| \: s_0b = s_0 \} \]

We note that if the action of $B$ on $S$ is transitive:

\[S  \simeq B/H = \{ H\tau_1, H\tau_2 \ldots \}\]

for some $\tau_1, \tau_2 \ldots \in B$.

\bigskip

Lastly, the following notation will be helpful.

\bigskip

\begin{definition} \label{a^s}

Let $a^{(s)} \in A^{(S)}$ be such that $a^{(s)}(s) = a$ and $a^{(s)}(t) = 1_A$ if $t \neq s$.
    
\end{definition}

We note that:

\[b^{-1}a^{(s_0)}b = a^{(s_0 b)} .\]

\section{Presentation of Wreath Products} \label{pres}

  We now consider the group presentation of unrestricted wreath products.  This section is based on the work of Yves de Cornulier in \cite{C3}.

\bigskip

\begin{definition}

For two elements $g,h \in G$, $[g,h] =g^{-1}h^{-1}gh$ is the commutator.

\end{definition}

In what follows we assume that $A$ and $B$ are both Class $\mathcal{C}$ and therefore recursively presented groups, so:

\[A = \langle a_1 \ldots a_m \: | \: r(a_1 \ldots a_m)=1_A \: \forall r \in R \rangle \]

and: 

\[B = \langle b_1 \ldots b_n \: | \: t(b_1 \ldots b_n)=1_B \: \forall t \in T \rangle \]  

where $n$ and $m$ are finite and $R$ and $T$ are the sets of relators on $A$ and $B$ respectively. 

\bigskip

\begin{lem} \label{pres2}

If both $A$ and $B$ are Class $\mathcal{C}$, the action of $B$ on $S$ is transitive, and we can solve the membership problem for $H$, then $A \operatorname{wr}_S B$ may be recursively presented, with the presentation:

\[A \operatorname{wr}_S B =  \langle A, B \: | \: [H,A] = 1_A, [A, A^b] = 1_A \; \forall b \in B - H \rangle . \]

In particular, in $A \operatorname{wr} B$, where $H$ is trivial:

\[A \operatorname{wr} B = \langle A, B \: | \: [A, A^b] = 1_A \; \forall b \in B, \:b \neq 1_B \rangle .\]

\end{lem}

\begin{proof}

As the action of $B$ on $S$ is transitive, if $H = Stab_B(s_0)$, $S = \{H \tau_1, H\tau_2,  \ldots \}$ with $\tau_1, \tau_2  \ldots \in B$.  By assumption, for any $b \in B$ we can determine whether or not $b \in H$.

If $b \in H$ then $b H = H b = H $ and so $H$ commutes with $A$.  

If $b \notin H$, for all $a_1, a_2 \in A$ $a_1$ commutes with $b^{-1}a_2b$ so $[A, A^b] = 1_A$.  

The relations in the presentation of $A \operatorname{wr}_S B$ are therefore:

\begin{align*}
r(a_1 \ldots a_m) &= 1_A \; \forall r \in R\\
t(b_1 \ldots b_p) &= 1_B \; \forall t \in T \\
ba  &= ab, \: \forall b \in H, a \in A \\
[a_1, a_2^b] &= 1_A, \: \forall a_1, a_2 \in A, b \in B - H \\
\end{align*}

This gives us the following presentation of $A \operatorname{wr}_S B$: 

\[ \langle A, B \: | \: [H,A] = 1_A, [A, A^b] = 1_A \; \forall b \in B - H \rangle \]

As $A$ is Class $\mathcal{C}$, we can determine whether any element of $A$ is $1_A$.  Therefore the above represents a recursive presentation of $A \operatorname{wr}_S B$.  

In $A \operatorname{wr} B$, $H$ is trivial and so the presentation of $A \operatorname{wr} B$ is: 

\[ \langle A, B \: | \: [A, A^b] = 1_A \; \forall b \in B, \:b \neq 1_B \rangle\]

\end{proof}

We apply the above presentation for $A wr_S B$ in the following Lemma.

\bigskip

\begin{lem} \label{support}

For any $f \in A^{(S)}$ the elements of $\sigma(f)$ are determinable.
    
\end{lem}

\begin{proof}

  Let $\{a_i\} $ and $\{b_j\}$ be the finite set of generators of $A$ and $B$ respectively. As the action of $B$ on $S$ is transitive (and recalling Definition \ref{a^s}), any element of $f$ may be written as:
\begin{align*}
f &= a_{i_1}^{(s_0)}b_{j_1}a_{i_2}^{(s_0)}b_{j_2} \ldots a_{i_l}^{(s_0)}b_{j_l} \\
&= b_{j_1}b_{j_2}(b_{j_2}^{-1}a_{i_1}^{(s_0)}b_{j_2})a_{j_2}^{(s_0)} \ldots a_{i_l}^{(s_0)}b_{j_l} \\
&= b_{j_1}b_{j_2}a_{i_1}^{(s_0 b_{j_2})}b_{j_3}a_{j_2}^{(s_0 b_{j_3})} \ldots a_{i_l}^{(s_0)}b_{j_l} \\
&= b_1b_2 \ldots b_l a_{i_1}^{(s_0b_{j_2} \ldots b_{j_l})} a_{i_2}^{(s_0b_{j_3} \ldots b_{j_l})} \ldots a^{(s_0 b_{j_l})}\
\end{align*}

where we rely on the fact that the membership problem in $H$ is solvable, so it is possible to determine whether $s_0x = s_0y$ for any $x,y \in B$, and if $b \notin H$, then $[A, A^b] = 1_A$. If $s_0x \neq s_0y$ then $a^{(s_0x)}$ and $a^{(s_0y)}$ commute, and as the word problem in $A$ is solvable, it is possible to determine for which $s \in S$ $f(s) \neq 1_A$ and therefore the elements of $\sigma(f)$. 

\end{proof}

\section{Matthews Maps} \label{Part1}

In the remainder of this paper, we deal only with the restricted wreath products $A \wr B$ or $A \wr_S B$.

\subsection{Background}

In \cite{M1} Matthews takes an arbitrary element $bf$ in the standard restricted wreath product $ A \wr B$, divides $B$ into its $\langle b \rangle$ cosets and then maps each coset to a single element of $A$.  We will extend these maps to the permutational restricted wreath product $A \wr_S B$.  It is first worthwhile pausing to discuss the idea behind these maps in a little more detail.

We start with a simple example, by assuming that $B = Sym_5$ and $S = \{1,2,3,4,5\}$.  We take $bf \in A \wr_S B$ with $f=(a_1,a_2,a_3, a_4, a_5) \in A^{(S)}$ and $b = (13)(45)$.  There are therefore three $b$ orbits in $bf$, $O_1=\{1,3\}$ , $O_2 = \{4,5\}$ and $O_3=\{2\}$. 

Let $h = (1_A,1_A, a_3, a_4, 1_A)$ and so $bf' \sim_{A\wr_S B} bf$ where:

\begin{align*}
bf' &= (1_Bh)(bf)(1_Bh)^{-1} \\
&= (b, (a_3a_1, a_2, 1_A, 1_A, a_4a_5)) \\
\end{align*}

In other words, we can conjugate $bf$ by an element of $A^{(S)}$ so that all the elements of each $b$-orbit are concentrated in $f'$ as the entry for one element of the $b$-orbit, with all other entries for that orbit in $f'$ being the identity.

By contrast, if instead we define $h = (a_1,1_A, 1_A, 1_A, a_5)$ then:

\[bf' = (b, (1_A, a_2, a_1a_3, a_5a_4, 1_A )) \]

In the first example above, the non-identity element for $O_1$ is $a_3a_1$ whilst in the second it is $a_1a_3$. From this we can see that, if we conjugate in this way, then, unless $A$ is abelian, the element for an orbit is not unique (as it depends on the choice of conjugator).  

In \cite{M1} Matthews uses this technique to test is $bf \sim_{A \wr B} cg$, by mapping each $\langle b \rangle$ coset to a single element of $A$ (accepting that this element may not be unique) and then using these maps to compare the $\langle b \rangle$ cosets in $f$ with the $\langle c \rangle$ cosets in $g$.

\subsection{Order and Orbit Order} \label{OO}

Matthews defines the maps differently depending on whether the order of $b$ is finite or infinite.  She states that we can determine the order of $b$ because $B$ has a solvable power problem.  We do not think that this is correct.  This is also mentioned in \cite{M2}.

Matthews defines the solvable power problem as the ability to decide for any two elements $x$ and $y$ in $B$ whether there exists an integer $n$ such that $y = x^n$.  This  may not allow us to determine which elements of $B$ have finite order (see Collins Theorem C \cite{C2}). In fact, as will be seen later, the distinction between finite and infinite order does not matter if $A$ is abelian.  We do, however, need to be able to determine whether the order of $b$ is finite if $A$ is not abelian. 
 This additional condition does require an additional Corollary \ref{orderproof}, to complete the necessity part of the proof of our revised version of Theorem \ref{B}. 

 In Lemma $2.5$ \cite{C2}, Collins gives an example of a finitely presented group $G$ with a solvable power problem but an unsolvable order problem (based on Example 1 \cite{M7} and relying on the existence of an injective map $\psi: \Z \rightarrow \Z$ (see \cite{B4}) where we can determine $\psi(n)$ but we cannot determine whether there exists an $m$ such that $\psi(m) = n$). If $A$ is a non-abelian group, then $A \wr G$ has an unsolvable conjugacy problem. 

Finally, we note that, if in the definition of the solvable power problem Matthews had required $n \neq 0$ then we could determine whether the order of $x$ was finite by setting $y =1_B$. This would not, however, be a sufficient change to Theorem B \cite{M1} (which is if and only if) because, as can be seen from Corollary \ref{orderproof} below, the condition for $n \neq 0$ is not required if $A$ is abelian. 

In $A \wr_S B$, the equivalent assumption is that we can solve the orbit order problem, by which we mean we can determine whether the number of elements of $S$ in a $b$-orbit $O_i$, which we will denote $|O_i|$, is finite or infinite. 

We note that each $b$-orbit $O_i$ can be expressed as the double coset $Ht_i\langle b \rangle$ for some $t_i \in B$ and that $H t_i \langle b \rangle = \bigcup_{i \in \Z} H t b^i$.  Then $|O_i| = |\{H t_i \langle b \rangle \} |$ is the number of right $H$ cosets in $\{Ht_i\langle b \rangle\}$).

\bigskip

\begin{lem} \label{OOint}

$|H t \langle b \rangle |$ is infinite if and only if $H \cap \langle tb t^{-1} \rangle = \{1_B\}$.  

\end{lem}

\begin{proof}

$|H t \langle b \rangle |$ is finite if and only if for some $j \in \Z$, $Htb^j = Ht$ or $Htb^jt^{-1} = H$.  Therefore, as cosets partition a group, if $Ht\langle b \rangle$ is infinite, $H \cap \langle t b t^{-1} \rangle = \{1_B\}$.  

\end{proof}

Therefore the orbit order problem is equivalent to determining for any $\beta \in B$ whether $H \cap \langle \beta \rangle =\{1_B\}$. 

\subsection{Defining the Matthews Maps}

We now define the Matthew Maps more carefully and at the same time extend them to $A\wr_SB$, making the following assumptions: 

\begin{enumerate}[(i)]
    \item $A$ and $B$ are non-trivial recursively presented groups; 
    \item $B$ acts on a set $S = B/H$ transitively;
    \item the conjugacy problem is solvable in $A$ and in $B$; and
    \item either $A$ is abelian or the orbit order problem is solvable in $B$;
    \item the membership problem for $H \gamma \langle \beta \rangle$ (for any $\gamma, \beta \in B$) is solvable.     
    \end{enumerate}

    We note that, if the membership problem for $H \gamma \langle \beta \rangle$ is solvable, by setting $\gamma = \beta = 1_B$, the membership problem for $H$ is solvable and, if $\beta \tau^{-1} \in H$ then $\beta \in H\tau$.

    \bigskip
    
\begin{lem} \label{sameorbit}

For any two points in $s_1, s_2 \in S$ if we can solve the membership problem for $H \gamma \langle b \rangle$ we can determine if $s_1$ and $s_2$ are in the same $b$-orbit.

\end{lem}  

\begin{proof}    
    
As $S = B/H$ let $s_1 = Ht_1$ and $s_2 = Ht_2$, $t_1, t_2 \in B$.  By assumption, the membership problem for $H \gamma \langle b \rangle$ for any $\gamma \in B$ is solvable. Therefore we can determine if $t_2 \in Ht_1\langle b \rangle$ and therefore whether $s_1$ and $s_2$ are in the same $b$-orbit, such that there exists $m,n \in \Z$ such that $s_1b^m = s_2b^n$.  

\end{proof}

For any $f \in A^{(s)}$, by Lemma \ref{support} we can determine the elements of $\sigma(f)$ and by Lemma \ref{sameorbit} we can partition $\sigma(f) \subseteq B/H$ into $b$-orbits.  We will now take $\{O_i\}$ to be the set of $b$-orbits that contain an element of $\sigma(f)$, which we will refer to as the non-trivial $b$-orbits.  Let $T = \{t_i\}$ being a set of representatives of $\{O_i\}$, where each $t_i \in B$.  In other words, we partition $\sigma(f)$ into its $b$-orbits and take one representative from each non-trivial $b$-orbit so that any element of $\sigma(f)$ may be expressed uniquely as $s_0t_i b^j $ for some $t_i \in T$ and some $j \in \Z$.

Having partitioned $\sigma(f)$ into the $b$-orbits $O_i$ each with representative $t_i$, each  $\beta \in B$ may be associated with a class of maps $A^{(S)} \rightarrow A$, $O_i \mapsto f_i^{*\beta}$ where $f_i^{* \beta}$ is defined below.  

We first note that here we are relying on the fact that $\sigma(f)$ is finite for two reasons.  First, as $\sigma(f)$ is finite,  $\{O_i\}$ is a finite set.  Secondly, for any $f$ and $\beta$ there exist integers $M$ and $N$ (with $N \geq 0 $) such that if $j < M$ or $j \geq M+N$ then $f(s_0t_ib^j\beta^{-1})= 1_A$. If $A$ is abelian, by our assumptions we may not know whether the $b$-orbit is finite or infinite.  We can, however, test whether $s_0t_ib^j = s_0t_i$ for all $j \leq N$ and therefore determine whether $|O_i|$ is less than or equal to $N$ and, if it is, to determine $|O_i|$. We also rely on this in Lemma \ref{inf1} below.

\bigskip
  
\begin{definition}

We define $f_i^{*\beta} $ as follows:

\begin{enumerate}[(i)]
\item if $A$ is not abelian and $|O_i| = n$ finite: 
\[ f_i^{*\beta} = \displaystyle \prod_{j=0}^{n-1} f( s_0t_i b^j \beta ^{-1}) \in A\]
\item if $A$ is not abelian and $|O_i|$ is infinite:
\[ f_i^{*\beta} = \displaystyle \prod_{j=M}^{M+N} f(s_0t_i b^j \beta ^{-1}) \in A\]
\item if $A$ is abelian and $|O_i| = n\leq N$:
\[ f_i^{*\beta} = \displaystyle \prod_{j=0}^{n-1} f( s_0t_i b^j \beta ^{-1}) \in A\]
\item if $A$ is abelian and $|O_i| > N$ (and may be finite or infinite):
\[ f_i^{*\beta} = \displaystyle \prod_{j=M}^{M+N} f(s_0t_i b^j \beta ^{-1}) \in A\].
\end{enumerate}

\end{definition}

We note that, if we are in (i) above, then $f_i^{*\beta}$ depends on the choice of representative $t_i$ for the orbit $O_i$, as a different choice would have cyclically permuted the factors in the product. If $A$ is abelian or $|O_i|$ is infinite, $f_i^{*\beta}$ does not depend on the choice of representative.

In \cite{M1}, Matthews does not give any background to the reasons behind this choice of maps, but the following Lemma may be insightful. 

\bigskip

\begin{lem} \label{RSC}

Let $bf \in A \wr_S B$.  Then it is possible to find $\bar{f} \in A^{(S)}$ such that $bf \sim b\bar{f}$ and $\sigma(\bar{f})$ contains one element $s_i$ in each non-trivial $b$-orbit $O_i$, with $\bar{f}(s_i) = f_i^*$.
        
\end{lem}

\begin{proof}

For any non-trivial $b$-orbit $O_i$ in $bf$, by Lemmas \ref{support} and \ref{sameorbit} we can determine:

\[\mu_i = |O_i \cap \sigma(f)|\]

so $\mu_i$ is the number of non-identity elements in the $b$-orbit $O_i$.  As $\sigma(f)$ is finite, there are only a finite number of $O_i$ where $\mu_i >0$.  We note that in $b\bar{f}$, $\mu_i \leq 1$ for all $O_i$.  

It is possible to perform a \textit{right shift conjugation} on $O_i$ defined as follows.  

If $\mu_i \geq 1$, let:

\[O_i \cap \sigma(f) = \{s_0t_ib^{k_1}, \ldots s_0t_ib^{k_n} \} \neq \emptyset\]

with $k_1 < k_2 <\ldots <k_n$ and let $ f(s_0tb^{k_p}) = a_p$.  We note that if $|O_i|$ is infinite then there is only one choice of $\{k_p\}$ but if $|O_i|$ is finite this is not the case.  We discuss this further below.

We take $h \in A^{(S)}$ such that $h = (a_1^{-1})^{(s_0t_ib^{k_1})}$.  Then:

\[bf \sim (1_B,h)^{-1}(bf)(1_B,h) = (b, (h^{-1})^bfh) \]

where:

\begin{align*}
(h^{-1})^bfh(s_0tb^{k_1}) &= 1_A a_1 a_1^{-1} = 1_A \\ 
(h^{-1})^bfh(s_0tb^{k_1+1}) &= a_1 f(s_0tb^{k_1+1}) \\ 
(h^{-1})^bfh(s) &= f(s) \: \operatorname{if} s \neq s_0tb^{k_1}, s_0tb^{k_1+1}
\end{align*}

If $O_i$ is such that $|O_i \cap \sigma(f)| = \mu_i >1$ then the effect of the right shift conjugation is as follows:

\begin{enumerate}[(i)]
    \item if $f(s_0tb^{k_1+1}) = a_1^{-1}$ then $\mu_i $ is reduced by two;
    \item if we are not in (i) but $k_2 = k_1 +1$ then then $\mu_i $ is reduced by one and $ f(s_0tb^{k_2}) = a_1a_2$; 
    \item if $\mu_i >1$ but we are not within (i) or (ii) above (so $f(s_0tb^{k_1+1}) = 1_A$) then $\mu_i$ remains unchanged and $(h^{-1})^bfh(s_0tb^{k_1+1})  = a_1$ and so $s_0t_ib^{k_1+1}$ replaces $s_0t_ib^{k_1}$ in the support;
    \item if $\mu_i = 1$, then $\mu_i$ remains unchanged and $s_0t_ib^{k_1+1}$ replaces $s_0t_ib^{k_1}$ in the support; or
    \item if $\mu_i = 0$, then $\mu_i$ remains unchanged and right shift conjugation has no effect on $O_i$.
     \end{enumerate}

Right shift conjugation on $O_i$ has no effect on any other $b$-orbit. 

We note that, if (iii) above applies and $\mu_i>1$, further applications of right shift conjugation will eventually bring us within either (i) or (ii) above, so reducing $\mu_i$.  Therefore repeated applications of right shift conjugation to $O_i$ will eventually reduce $\mu_i$ to either $0$ or $1$, but will leave the other orbits unchanged.

We first assume that $|\{Ht \langle b \rangle\}|$ is infinite.  We can see that, repeated applications of a right shift conjugation on $O_i$ will eventually conjugate $bf$ to $bf_i$ where $f_i \in A^{(S)}$ and:

\begin{align*}
f_i(s_0t_ib^{k_n}) &= \prod_{j=k_1}^{k_n} f(s_0t_ib^j) = f^*_i \\
f_i(s) &= 1_A, s \in O_i, s \neq s_0t_ib^{k_n} \\
f_i(s) &= f(s), s \notin O_i 
\end{align*}

From this point further applications of right shift conjugation on $O_i$ will simply change the element of $O_i$ in the support, but the non-identity element will remain $f_i^*$. 

Next we assume that $|O_i|$ is finite, so if:

\[O_i \cap \sigma(f) = \{s_0t_ib^{k_1}, \ldots s_0t_ib^{k_n} \} \neq \emptyset\]

then the choice of $\{k_p\}$ is not unique.  We take one choice of $\{k_p\}$ where $k_1 < k_2 <\ldots < k_n$.  Again, repeated applications of a right shift conjugation on $O_i$ will eventually conjugate $bf$ to $bf_i$ where:

\begin{align*}
f_i(s_0t_ib^{k_n}) &= \prod_{j=k_1}^{k_n} f(s_0t_ib^j) = f_i^* \\
f_i(s) &= 1_A, s \in O_i, s \neq s_0t_ib^{k_n} \\
f_i(s) &= f(s), s \notin O_i 
\end{align*}

We recall that, if $A$ is not abelian and $|O_i|$ is finite, $f_i^*$ is not unique, and here $f_i^*$ depends on the choice of $\{k_p\}$.  A different choice of $\{k_p\}$ would result in a cyclic permutation of the elements in the product in $f_i^*$.

Whether $|O_i|$ is finite or infinite, as right shift conjugation on $O_i$ has no effect on the other $b$ orbits, so we can perform right shift conjugation in turn on each orbit $O_i$ for which $\mu_i >1$.  By considering these repeated applications of right shift conjugation, we can see that $bf \sim b\bar{f}$ where $\sigma(\bar{f})$ contains at most one element $s_i \in S$ for each $b$ orbit $O_i$, with $\bar{f}(s_i) = f_i^*$.

\end{proof}

This gives us an indication of the reasons behind the Matthews approach to solving the conjugacy problem.  It is equivalent to reducing $bf$ and $cg$ to $b\bar{f}$ and $c\bar{g}$ respectively, and then comparing the single non-identity element in the $b$ and $c$ orbits.

We now demonstrate this technique more formally.

\bigskip

\begin{lem} \label{order}

Let $W = A\wr_S B$ and let $x = bf \sim_Wy = cg$ so there exists a conjugator $z =dh$, $d \in B$, $h \in A^{(S)}$ such that $zx=yz$. Let $\{O_i\}$ be the non-trivial $b$-orbits in $S$ and let $\{\Omega_i\}$ be the non-trivial $c$-orbits in $S$. Then $b \sim_B c$ and $|O_i| = |\Omega_i|$ where if $s_0t_i \in O_i$ then $s_0t_id^{-1} \in \Omega_i$.
    
\end{lem}

\begin{proof}

If $x=bf \sim_{A \wr_S B} cg = y $ then $zx=yz$ requires $(db, h^{b}f) = (cd, g^{d}h)$. Hence $db = cd$ and $b \sim_B c$.

Suppose $|O_i| = n$ (finite), so $s_0t_ib^n = s_0t_i$, and $t_ib^nt_i^{-1} \in H$.

But if $b = d^{-1}cd$ then $b^n = d^{-1}c^nd$ and so $t_id^{-1}c^ndt_i^{-1} \in H$ and $s_0t_id^{-1}c^n = s_ot_id^{-1}$.  Therefore  $|O_i|$ equals $|\Omega_i|$ where $s_0t_id^{-1} \in \Omega_i$.

Simiarly, if $|O_i|$ is infinite, then there is no $j \in \Z$ such that $s_0t_ib^j = s_0t_i$.  Therefore  $t_id^{-1}c^jdt_i^{-1} \notin H$ and $s_0t_id^{-1}c^j \neq s_ot_id^{-1}$ for all $j$.  Therefore $|\Omega_i|$ is infinite.

\end{proof}

A consequence of Lemma \ref{order} is that, if $db = cd$, $d$ takes $b$ orbits to $c$ orbits of the same orbit order.

\bigskip

 \begin{lem} \label{changeg}

Let $bf, cg \in A \wr_S B$.  If $b \sim_B c$ so $db = cd$ for some $d \in B$ then $cg$ is conjugate to $bg^d$.
    
\end{lem}

\begin{proof}

Noting that $(d1)^{-1} = (d^{-1} 1)$ we can see that:

\begin{align*}
(d1)^{-1}(cg)(d1) &= (d^{-1}1)(cd, g^d) \\
&=(d^{-1}cd, g^d) \\
&= bg^d
\end{align*}
    
\end{proof}

A consequence of Lemma \ref{changeg} is that, if $bf \sim_W cg$ then $bf \sim_W bg^d$. 

Putting this together, if $bf \sim_W cg$ then $b\bar{f} \sim_W b\bar{g}^d$ where $d$ takes $b$-orbits to $b$-orbits with the same orbit order. If the orbit order is infinite or $A$ is abelian, repeated applications of right (or left) shift conjugation eventually brings the non-identity element $g^{*d}$ to the same position in the relevant $b$ orbit in $f$.  If $A$ is not abelian and the orbit order is finite, the right shift or left shift conjugation will also conjugate $g^{*d}$.

We can prove that more formally as follows.

\bigskip

\begin{lem} \label{fin1}

Let $x = bf$, $y = cg$ and $z = dh$ with $x,y,z \in A \wr_S B$, $b,c,d \in B$ and $f,g,h \in A^{(S)}$.  If $zx = yz$ and $|O_i| = n$ finite then $g_i^{*d} \sim_A f^*_i$.
    
\end{lem}

\begin{proof}

If $zx = yz$, then $(db, h^bf) = (cd, g^dh)$ and so $bd= dc$ and:

\[h^b(s)f(s)) = g^d(s)h(s)) \]

for all $s \in S$.  Then:

\begin{align*}
g_i^{*d} &= \displaystyle \prod_{n=0}^{n-1} g(s_0 t_i b^j d ^{-1} ) \\
&= \displaystyle \prod_{j=0}^{n -1} h(s_0t_i b^{j-1})f(s_ot_ib^j)h^{-1}(s_0t_ib^j) \\
&= h(s_0t_i b^{-1}) \big{[} \displaystyle \prod_{j=0}^{n -1} f(s_0t_i b^j) \big{]} h^{-1}(s_0t_i b^{n-1}) \\
&= h(s_0t_ib^{-1}) f^*_i h^{-1}(s_0t_ib^{-1}) \\
\end{align*}

because in the product each of the intermediate terms $h(s_0t_ib^{j})h^{-1}(s_0t_i b^j)$ cancel out.

Hence $g_i^{*d} \sim_A f_i^*$ as required.

\end{proof}

We can generalise the reduction of the product used in the above proof as follows:

\bigskip

\begin{lem} \label{formula}
    
For any $M,N \in \Z$, $N \geq 0$:

\[\displaystyle \prod_{j=M}^{M+N} g(s_0t_i b^j d^{-1}) = h(s_0t_i b^{M-1}) \big{[} \displaystyle \prod_{j=M}^{M+N} f(s_0t_i b^j) \big{]} h^{-1}(s_0t_i b^{M+N}) \]

\end{lem}

We then turn to the infinite case.

\bigskip

\begin{lem} \label{inf1}

Let $x = bf$, $y = cg$ and $z = dh$ with $x,y,z \in A \wr_S B$, $b,c,d \in B$ and $ f,g,h \in A^{(S)}$.  If $zx = yz$ and $|O_i|$ is infinite then $g_i^{*d}=f^*_i$ for all $i$.
    
\end{lem}

\begin{proof}

As we are dealing with the restricted wreath product, $\sigma(f)$, $\sigma(g)$ and $\sigma(h)$ are all finite and we may choose integers $M$ and $N$ (with $N \geq 0 $) such that if $j < M$ or $j \geq M+N$ then:

\[f(s_0t_ib^j)=g(s_0t_ib^jd^{-1})= h(s_0t_ib^j) = 1_A  \] 

This means that:

\[ g_i^{*d} = \displaystyle \prod_{j=M}^{M+N} g(s_0 t_ib^jd ^{-1}) \]

and:

\[ f^*_i = \displaystyle \prod_{j=M}^{M+N} f(s_0 t_ib^j) \]

Applying Lemma \ref{formula}: 

\[\displaystyle \prod_{j=M}^{M+N} g( s_0t_i b^j d ^{-1} ) = h(s_0t_i b^{M-1}) \big{[} \displaystyle \prod_{j=M}^{M+N} f(s_0t_i b^j) \big{]} h^{-1}(s_0t_i b^{M+N}) \]

and, from the choice of $M$ and $N$, $h(s_0t_i b^{M-1}) = h^{-1}(s_0t_i b^{M+N}) = 1_A$.

Therefore $g_i^{*d} =f^*_i$ for all $i$, as required.

\end{proof}

We note here that, if $A$ is abelian, then $g_i^{*d} \sim_A f^*_i$ requires $g_i^{*d}= f^*_i$.  Therefore when applying Lemmas \ref{fin1} and \ref{inf1} if $A$ is abelian we do not need to distinguish between the orbits of finite and infinite order.  This gives us the following:

\bigskip

\begin{lem} \label{abel1}

Let $x = bf$, $y = cg$ and $z = dh$ with $x,y,z \in A \wr_S B$, $b,c,d \in B$ and $ f,g,h \in A^{(S)}$.  If $zx = yz$ and $A$ is abelian then $g_i^{*d}=f^*_i$ for all $i$.
    
\end{lem}

We apply Lemmas \ref{fin1}, \ref{inf1} and \ref{abel1} as follows:

\bigskip

\begin{lem} \label{infin2}

Let $x = bf$ and $y = cg$ be elements of $A \wr_S B$ where $A$ is abelian.  Then $x$ and $y$ are conjugate in $A \wr_S B$ if and only if there exists $d \in B$ such that $db=cd$ in $B$ and $g_i^{*d} =f^*_i$.

\end{lem}

\begin{proof}

We have proved necessity in Lemmas \ref{order} and \ref{abel1}.

Next we assume that there exists $d \in B$ such that $db=cd$ and, for all $i$, $g_i^{*d} =f^*_i$.  We need to construct an $h \in A^{(S)}$ such that $zx=yz$ for $z=dh$.  As $zx = (db, h^bf)$ and $yz = (cd, g^dh)$, the first terms are equal by assumption.

As $\sigma(f)$ and $\sigma(g)$ are finite there exists integers $M,N$ with $N \geq 0$ such that, if $j<M $ or $j \geq M+N$:

\[g(s_0t_ib^jd^{-1}) = f(s_0t_ib^j) = 1_A\]

Any $s \in S$ can be expressed as $s=s_0t_ib^k$ for some $t_i \in T$ and some $k \in \Z$.  If $k < M$ or $k>M+N$ let $h(s_0t_ib^k) = 1_A$ and otherwise let:

\[ h(s_0t_ib^k) = \big{[} \displaystyle \prod_{j=M}^k g( s_0t_i b^j d^{-1}) \big{]}^{-1} \big{[} \displaystyle \prod_{j=M}^k f(s_0t_ib^j) \big{]} \]

Therefore, if $k < M$ or $k>M+N$ then:

\[h^bf(s_0t_ib^k) = f(s_0t_ib^k) = g^d(s_0t_ib^kd^{-1})= g^dh(s_0t_ib^kd^{-1}) =1_A\]

If $k = M+N$ then:

\begin{align*}
h(s_0t_ib^{M+N}) &= \big{[} \displaystyle \prod_{j = M}^{M+N} g(s_0t_i b^j d ^{-1}) \big{]}^{-1} \big{[} \displaystyle \prod_{j=M}^{M+N} f(s_0t_i b^j) \big{]} \\
&= (g^{*d})^{-1} f^*\\
&=1_A \\
\end{align*}

If $M\leq k  < M+N$ then:

\begin{align*}
h^bf(s_0t_ib^k) &= h(s_0t_i b^{k-1})f(s_0t_i b^k)  \\
&= \big{[} \displaystyle \prod_{j = M}^{k-1} g(s_0t_i b^j d ^{-1}) \big{]}^{-1} \big{[} \displaystyle \prod_{j=M}^{k-1} f(s_0t_i b^j) \big{]} f(s_0t_i b^k)\\
&= g(s_0t_i b^k d^{-1}) \big{[} \displaystyle \prod_{j=M}^k g(s_0 t_i b^k d ^{-1}) \big{]}^{-1} \big{[} \displaystyle \prod_{j=M}^k f(s_0t_i b^k) \big{]} \\
&=g(s_0t_i b^k d^{-1}) h (s_0t_i b^k) \\
&=g^dh(s_0t_i b^k) \\
\end{align*}

\end{proof}

\bigskip

\begin{lem} \label{fin2}

Let $x = bf$ and $y = cg$ be elements of $A \wr_S B$ where $A$ is not abelian. Then $x$ and $y$ are conjugate in $A \wr_S B$ if and only if there exists $d \in B$ such that $db=cd$ in $B$, and for each $b$-orbit $O_i$:

    \begin{enumerate} [(i)]
    \item if $|O_i|$ is finite,  $g_i^{*d} \sim_A f^*_i$; and        \item if $|O_i|$ is infinite, $g_i^{*d} =f^*_i$. 
        
        \end{enumerate}

\end{lem}

\begin{proof}

We have proved necessity in Lemmas \ref{order}, \ref{fin1} and \ref{inf1}.

We assume that $A$ is not abelian and there exists $d \in B$ such that, $db=cd$ and for each $b$ orbit $O_i$:
    \begin{enumerate} [(a)]
     \item if $|O_i| = n$ finite,  $g_i^{*d} \sim_A f^*_i$; and        \item if $|O_i|$ is infinite, $g_i^{*d} =f^*_i$. 
       
        \end{enumerate}

We need to construct an $h \in A^{(S)}$ such that $zx=yz$ where $z =dh$.  As $zx = (db, h^bf)$ and $yz = (cd, g^dh)$, the first terms are equal by assumption.

If $|O_i| = n$ finite, there exists $\gamma_i \in A$ such that $g_i^{*d} = \gamma_i f^*_i \gamma_i^{-1}$. For $0 \leq k <n$ let:

\[ h(s_0t_i b^k) = \big{[} \displaystyle \prod_{j =0}^k g(s_0t_i b^j d^{-1}) \big{]}^{-1} \gamma _i \big{[} \displaystyle \prod_{j =0}^k f(s_0t_i b^j) \big{]} \]

Then:

\begin{align*}
h^bf(s_0t_i b^k) &= h(s_0t_i b^{k-1}) f(s_0t_i b^k) \\
&= \big{[} \displaystyle \prod_{j =0}^{k-1} g(s_0t_i b^j d ^{-1}) \big{]}^{-1} \gamma_i \big{[}  \displaystyle \prod_{j = 0} ^{k-1} f(s_0t_i b^j) \big{]} f(s_0t_i b^k)\\
&= g(s_0t_i b^j d^{-1}) \big{[} \displaystyle \prod_{j =0}^ k g(s_0 t_i b^j d ^{-1}) \big{]}^{-1} \gamma_i \big{[} \displaystyle \prod_{j =0} ^k f(s_0t_i b^j) \big{]} \\
&=g(s_0t_i b^j d^{-1}) h(s_0t_i b^k) \\
&=g^dh(s_0t_i b^k)
\end{align*}

If $|O_i|$ is infinite, as $\sigma(f)$ and $\sigma(g)$ are finite there exists integers $M,N$ with $N \geq 0$ such that, if $j<M $ or $j \geq M+N$:

\[g(s_0t_ib^jd^{-1}) = f(s_0t_ib^j) = 1_A\]

Let:

\[ h(s_0t_i b^k) = \big{[} \displaystyle \prod_{j =M}^k g(s_0t_i b^j d^{-1}) \big{]}^{-1} \big{[} \displaystyle \prod_{j =M}^k f(s_0t_i b^j) \big{]} \]

If $k <M$ then 

\[h^bf(s_0t_i b^k)= g^dh(s_0t_i b^k) = 1_A,\]

if $k > M+N$ then $h(s_0t_i b^k) = h^b(s_0t_i b^k)= 1_A$ and so:

\[h^bf(s_0t_i b^k)= f^* = g^{*d} = g^dh(s_0t_i b^k)\]

and otherwise:

\begin{align*}
h^bf(s_0t_i b^k) &= h(s_0t_i b^{k-1}) f(s_0t_i b^k) \\
&= \big{[} \displaystyle \prod_{j = M}^{k-1} g(s_0t_i b^j d ^{-1}) \big{]}^{-1}  \big{[}  \displaystyle \prod_{j = M} ^{k-1} f(s_0t_i b^j) \big{]} f(s_0t_i b^k)\\
&= g(s_0t_i b^k d^{-1}) \big{[} \displaystyle \prod_{j M}^ k g(s_0 t_i b^j d ^{-1}) \big{]}^{-1} \big{[} \displaystyle \prod_{j =M} ^k f(s_0t_i b^j) \big{]} \\
&=g(s_0t_i b^j d^{-1}) h(s_0t_i b^k) \\
&=g^dh(s_0t_i b^k)
\end{align*}
\end{proof}

    \section{Compliant Permutations of Orbits}

   \begin{definition}
Let $C_B(b)$ be the centraliser of $b$ in $B$, so:

\[C_B(b) = \{\beta \in B \:| \: b\beta = \beta b \} \]

\end{definition}

\bigskip

\begin{lem} \label{centraliser}

Let $\Gamma$ be the set of $b$-orbits on $S$, so $\Gamma  = \{ O_i\}$ where, if $s_1 , s_2 \in O_i$ then
$s_1b_j = s_2b_k$ for some $j,k \in \Z$. If $ d \in C_B(b)$ then $d$ acts on the set $\Gamma$ by mapping $b$-orbits
to $b$-orbits.    
\end{lem}

\begin{proof}

If $d \in C_B(b)$, and $s_1, s_2 \in S$ are in the same $b$-orbit then for some $j,k \in \Z$:

\begin{align*}
s_1b^j &= s_2b^k \\
\Leftrightarrow s_1b^jd &= s_2b^kd \\
\Leftrightarrow s_1db^j &= s_2db^k    
\end{align*}

and if $s_1b^j = s_1$ then $s_1db^j = s_1d$ so $d$ maps $b$-orbits to $b$-orbits of the same orbit order. 
    
\end{proof}

We consider $b\bar{f}, b\bar{g} \in A\wr_S B$, as defined in Lemma \ref{RSC}. As both $\sigma(f)$ and $\sigma(g)$ are finite, there can only be a finite number of non trivial $b$-orbits (where $f^*$ or $g^*$ is not trivial). 

We define compliant permutation below, but in essence it is a pairing of each of the non-trivial $b$-orbits in $\bar{f}$ with one of the non-trivial $b$-orbits in $\bar{g}$ where each pair satisfies the requirements of Lemmas \ref{infin2} or \ref{fin2} (so $f^* = g^*$ or $f^* \sim_A g^*$ as appropriate).  We note that such a pairing is only possible if there are the same number of non-trivial orbits in $\bar{f}$ and $\bar{g}$, and there can only be a finite number of compliant permutations.

For the avoidance of doubt, we are not at this stage asking if there is an element of $B$ that effects a compliant permutation, only whether a compliant permutation exists.

We define a compliant permutation more formally as follows.

 \bigskip

\begin{definition}

Let $b\bar{f}, b\bar{g} \in A\wr_S B$, let $\{O_i\}$ be the finite set of $b$-orbits in $S$ that contain an element of $\sigma(\bar{f})$ and let $\{\Omega_l\}$ be the finite set of $b$-orbits in $S$ that contain an element of $\sigma(\bar{g}) $.  A compliant permutation of $\{O_i\}$ is a bijection $\psi: \{O_i\} \rightarrow \{\Omega_l\}$ such that $\psi: O_i \mapsto \Omega_{\psi(i)}$ where:

\begin{enumerate} [(i)]
    \item if $A$ is not abelian and $|O_i|$ is finite, $|\Omega_{\psi(i)}|$ is finite and  $f_i^* \sim_A g_{\psi(i)}^*$; and
    \item if $A$ is not abelian and $|O_i|$ is infinite, $\Omega_{\psi(i)}$ is infinite and $f_i^* = g_{\psi(i)}^*$; and
     \item if $A$ is abelian, $f_i^* = g_{\psi(i)}^*$.    \end{enumerate}    

\end{definition}

We note that, under this definition, if $A$ is abelian, we do not need to be able to determine the orbit order.

It is convenient to record the following as Lemmas.

\bigskip

\begin{lem} \label{cpexists}

As $\{O_i\}$ and $\{\Omega_l\}$ are both finite sets, there can only be a finite number of compliant permutations and it is possible to determine all compliant permutations. If $bf \sim_W bg$ then a compliant permutation exists.

\end{lem}

\bigskip

\begin{lem} \label{n=m}

If $\{O_i\}$ and $\{\Omega_l\}$ contain a different number of orbits, then no compliant permutation exists.

\end{lem}

\bigskip

\begin{lem} \label{CPcent}

If there is an element $d \in B$ that effects a compliant permutation, then $d \in C_B(b)$.
    
\end{lem}

\section{Proof of Theorem 1.0.3} \label{perm}

\subsection{Proof of Theorem 1.0.2}

This is also the proof of sufficiency in Theorem \ref{conditions}.  

We begin with the case where $A$ is abelian.

 \bigskip
\begin{prop} \label{suffabel}

Let $A$ and $B$ are non-trivial recursively presented groups, where $A$ is abelian and $B$ acts on a set $S$ transitively, so $S = B/H$ where $H$ is the stabiliser of some arbitrary point $s_0 \in S$.   If:

\begin{enumerate}[(i)]
\item the conjugacy problem is solvable in $A$ and in $B$;
\item the membership problem for $H \gamma \langle \beta \rangle$ (for any $\gamma, \beta \in B$) is solvable; and
\item for any $\beta \in B$ and any finite set of $n$ pairs of elements $(\alpha_i, \gamma_i)$, where $\alpha_i, \gamma_i \in B$ and $1 \leq i \leq n$, we can determine whether or not $\big{[}\bigcap_{i=1}^n \alpha^{-1}_iH\gamma_i \langle \beta \rangle \big{]}  \cap C_B(\beta) = \emptyset $.    \end{enumerate} 

then the conjugacy problem is solvable in $W = A \wr_S B$.

\end{prop} 

 \begin{proof} 

We take arbitrary $bf, cg \in W$ and seek to determine whether $bf \sim_W cg$.

By assumption, we can solve the conjugacy problem in $B$.  If $b \nsim_B c$ then by Lemma \ref{order} $bf \nsim_W cg$ and we are done.

If $b \sim_B c$ then there exists $d \in B$ such that $db = cd$ and by Lemmas \ref{RSC} and \ref{changeg}, $bf \sim_W b\bar{f}$ and $cg \sim_W b\bar{g}^d$.  So if $bf \sim_W cg$ then $b\bar{f} \sim_W b\bar{g}^d$. We can therefore assume $b=c$.

By Lemma \ref{support} we can list the elements of $\sigma(f)$ and $\sigma(g)$.  As $\sigma(f)$ is finite, and (by assumption) we can solve the membership problem for $H \gamma \langle \beta \rangle$, applying Lemma \ref{RSC} we can construct $\bar{f}$ such that $\sigma(\bar{f}) = \{H\alpha_1, \ldots H\alpha_m\}$ (with $s_0 \alpha_i$ being a representative element of the $b$-orbit $O_i$) and $\bar{f}(H\alpha_i) = f_i^*$.   $\sigma(\bar{f})$ therefore contains one element from each non-trivial $b$-orbit $O_i = H \alpha_i \langle b \rangle$ and $\{O_i\}$ is a finite set of $m$ elements.

Similarly we can construct $\bar{g}$ such that $\sigma(\bar{g}) = \{H\gamma_1, \ldots H\gamma_n\}$, with $\bar{g}(H\gamma_l) = g_l^*$ and $\sigma(\bar{g})$ contains one element from each non-trivial $b$-orbit $\Omega_l = H \gamma_l \langle b \rangle$ and $\{\Omega_l\}$ is a finite set of $n$ elements.  

If $\bar{f} = 1$, (so $f^*_i = 1_A$ for all $b$-orbits), then $b\bar{f} \sim_W b\bar{g}$ if and only if $\bar{g} = 1$, which we can solve as by assumption we can solve the conjugacy problem and therefore the word problem in $A$.  

If $\bar{f}, \bar{g} \neq 1$, but $b\bar{f} \sim_W b\bar{g}$ then, by Lemma \ref{cpexists}, a compliant permutation of the $b$ orbits must exist and so, by Lemma \ref{n=m}, $n=m$.  Therefore if $n \neq m$ then $bf \nsim_W bg$ and we are done.  If $n=m$ but no compliant permutation exists, then $bf \nsim_W bg$ and we are done.

If we have got this far, we have a finite number of possible compliant permutations of the $b$-orbits, and applying Lemma \ref{CPcent} for each such compliant permutation we need to determine whether there is an element of $d \in C_B(b)$ that effects it.
 
 We note that if $O_i = H \alpha_i\langle b \rangle$ is mapped by $d$ to $\Omega_{\psi(i)} = H \gamma_i \langle b \rangle$ then $H \alpha_i d = H \gamma_i b^j$ for some $j \in \Z$.  As this needs to be satisfied for all $i$ this requires: 

 \[d \in \bigcap_{i = 1}^n \alpha_i^{-1}H\gamma_i \langle b \rangle\]
 
 Therefore, to solve the conjugacy problem in $A \wr_S B$, for any compliant permutation we need to be able to determine if: 
 
 \[\big{[}\bigcap_{i=1}^n \alpha^{-1}_iH\gamma_i \langle b \rangle \big{]}  \cap C_B(b) = \emptyset\]

 \end{proof}

Next we deal with the case where $A$ is not abelian.

\bigskip

\begin{prop} \label{suffnotab}

Let $A$ and $B$ are non-trivial recursively presented groups, where  $A$ is not abelian and $B$ acts on a set $S$ transitively, so $S = B/H$ where $H$ is the stabiliser of some arbitrary point $s_0 \in S$.   If:

\begin{enumerate}[(i)]
\item the conjugacy problem is solvable in $A$ and in $B$;
\item we can solve the orbit order problem in $B$;
\item the membership problem for $H \gamma \langle \beta \rangle$ (for any $\gamma, \beta \in B$) is solvable; and
\item for any $\beta \in B$ and any finite set of $n$ pairs of elements $(\alpha_i, \gamma_i)$, where and $1 \leq i \leq n$, $\alpha_i, \gamma_i \in B$ and $|H\alpha_i\langle b \rangle|$ is infinite, we can determine whether or not:

\[\big{[}\bigcap_{i=1}^n \alpha^{-1}_iH\gamma_i \langle \beta \rangle \big{]}  \cap C_B(\beta) = \emptyset \]  
\end{enumerate} 

then the conjugacy problem is solvable in $W = A \wr_S B$.

\end{prop}  

 \begin{proof} 

We take arbitrary $bf, cg \in W$ and seek to determine whether $bf \sim_W cg$.  As in Proposition \ref{suffabel} above, $b \nsim_B c$ then by Lemma \ref{order} $bf \nsim_W cg$ and if  $b \sim_B c$ we may therefore assume $b=c$. 

If $\bar{f} = 1$, (so $f^*_i = 1_A$ for all $b$-orbits), then $b\bar{f} \sim_W b\bar{g}$ if and only if $\bar{g} = 1$, which we can solve.  

In order for there to be a compliant permutation of the $b$-orbits, $\{O_i\} $ and $\{\Omega_l\}$ must contain the same number of orbits of each orbit order.  If this is not the case, then no compliant permutation exists, so $bf \nsim_W cg$ and we are done.

Otherwise, we know that under a compliant permutation $|H \alpha_i \langle b \rangle | = |H \gamma_i \langle b \rangle | = n_i$. We assume that $|O_i|$ is finite for $1 \leq i \leq m$ and that $|O_i|$ is infinite for $m+1 \leq i \leq n$.  

We are looking for $d \in B$ that maps $O_i$ to $\Omega_i$ for all $i$.  We assume that $|O_i| = n_i$ so $H \alpha_id= H\gamma_ib^{j_i}$ for some $1 \leq j_i \leq n_i$.  Therefore for any compliant permutation we need to be able to determine if: 
 
 \[\big{[}\bigcap_{i=1}^n \alpha^{-1}_iH\gamma_i \langle b \rangle \big{]}  \cap C_B(b) = \emptyset\]

where we can rewrite the left hand side as:

\[\big{[} \bigcap_{i=1}^m \alpha_i^{-1} H \gamma_i \langle b \rangle \big{]} \cap \big{[}\bigcap_{i=m+1}^n \alpha^{-1}_iH\gamma_i \langle b \rangle \big{]}  \cap C_B(b) = \Gamma_f \cap \Gamma_\infty \]

where:

\[\Gamma_f = \big{[}\bigcap_{i=1}^m \alpha_i^{-1} H \gamma_i \langle b \rangle \big{]}\]

and

\[\Gamma_\infty = \big{[} \bigcap_{i=m+1}^n \alpha^{-1}_iH\gamma_i \langle b \rangle \big{]}  \cap C_B(b) \]

As in the abelian case, we can determine if $\Gamma_\infty$ is empty. If it is, then the Solution Set is empty and we are done.  In addition, if there are no orbits of finite order, so $m=0$, then as in the abelian case we can determine if $bf \sim_W cg$.

If $m=1$, $\Gamma_f = \alpha^{-1} H \gamma \langle b \rangle $.  If $|\alpha^{-1} H \gamma \langle b \rangle| = n_i$ the Solution Set becomes:

\[\big{[}\bigcap_{i=2}^n \alpha^{-1}_iH\gamma_i \langle b \rangle \big{]}  \cap C_B(b) \cap \alpha_1^{-1}H \gamma_1 b^j= \emptyset \]

for some $1 \leq j \leq n_i$ which by assumption we can solve for each possible value of $j$.  Hence we can determine if $bf \sim_W cg$.

If $m >1$ and $d \in \Gamma_f$ then $d \in \alpha_i^{-1}H \gamma_ib^{j_i}$ for each $i$ where $1 \leq j_i \leq n_i$.  Therefore as cosets partition the group:

\[d \in \alpha_i^{-1}H \gamma_ib^{j_i} =\alpha_k^{-1}H \gamma_kb^{j_k}\]

for any pair $(i,k)$ where $1 \leq i,k \leq m$.

This requires:

\[\alpha_k \alpha_i^{-1} \gamma_2b^{j_i - j_k}\gamma_k \in H\]

which we can test as we can solve the membership problem for $H$ and $j_i - j_k$ can take only a finite number of values.  If for any pair $(i,k)$ there is no possible value of $j_i -j_k$ such that $\alpha_k \alpha_i^{-1} \gamma_2b^{j_i - j_k}\gamma_k \in H$ then $\Gamma_f = \emptyset$ and $bf \nsim_W cg$.

If there are possible values of $j_i -j_k$ then we can also find all possible values of $j_i$ and $j_k$. We can therefore list the finite number of elements in $\Gamma_f$ and for each test if its intersection with $\Gamma_\infty$ is empty.  Hence we can determine if $bf \sim_W cg$.

\end{proof}

 In the remainder of this paper we will refer to $\big{[} \bigcap_{i=1}^n \alpha_i^{-1}H\gamma_i \langle b \rangle \big{]} \cap C_B(b)$ as the \textit{Solution Set} for a compliant permutation.

\subsection{Proof of Necessity in Theorem 1.0.3} \label{neces}

The following Lemmas are needed in the proofs of necessity in Theorem \ref{conditions}.

\bigskip

\begin{lem} \label{orbitorder}

Let $A$ and $B$ be non-trivial groups where $A$ is not abelian and $B$ acts on a set $S \neq \emptyset$ transitively such that $S = B/H$ for some $H= Stab_B(s_0)$ for some $s_0 \in S$.  If the conjugacy problem in $A \wr_S B$ is solvable then for any $b \in B$ and any $b$-orbit $O_i$ we can determine $|O_i|$.

\end{lem}

\begin{proof}

Let $O_i$ be a $b$-orbit in $S$ with representative $t_i$, so $O_i = Ht_i \langle b \rangle$.

As $A$ is not abelian, we take $a_1, a_2  \in A$, $a_1 \sim a_2$ but $a_1 \neq a_2$.  Let $f = a_1^{(t_i)}$ and $g = a_2^{(t_ib)}$ so  $f_i^* = a_1$ and $g_i^* = a_2$.  Applying Lemma \ref{fin2}, $|O_i|$ is finite if and only if $bf$ and $bg$ are conjugate in $A \wr_S B$, which by our hypothesis is solvable.  

We note that, if it is determined that $|O_i|$ is finite, testing $s_0t_ib^j$ for increasing values of $j$ will eventually give $|O_i|$.

\end{proof}

Applying this with $H$ trivial and so $S = B$ gives us the following (which is required for Theorem \ref{B}):

\bigskip

\begin{cor} \label{orderproof}

If $A$ and $B$ be non-trivial groups where $A$ is not abelian and the conjugacy problem in $A \wr B$ is solvable then we can determine the order of $b$. 

\end{cor} 

Next we consider the membership problem in $H \gamma \langle b \rangle$.

\bigskip

\begin{lem} \label{coset}

Let $A$ and $B$ be non-trivial groups where $A$ has an infinite number of conjugacy classes and the action of $B$ on a set $S \neq \emptyset$ is transitive such that $S = B/H$ for some $H= Stab_B(s_0)$ for some $s_0 \in S$. If the conjugacy problem in $A \wr_S B$ is solvable then the membership problem for $H \gamma \langle \beta \rangle$ (for any $\gamma, \beta \in B$) is solvable. 

\end{lem}

\begin{proof}

We note that the membership problem for $H \gamma \langle \beta \rangle$ asks, for any $w \in B$, whether $w \in H \gamma \langle \beta \rangle$.  This is equivalent to determining whether $s_1 = s_0w$ and $s_2= s_0\gamma$  are in the same $b$-orbit.

Take $a \in A$, $a \neq 1_A$.  Let $f = a^{(s_1)} + (a^{-1})^{(s_2)}$.  From Lemmas \ref{infin2} and \ref{fin2} above $bf \sim_{A \wr_SB} (b,1)$ if and only if $f^*_i = 1_A$ for all $i$ which requires $s_1$ and $s_2$ to be in the same $b$-orbit.

\end{proof}

Next we consider the conjugacy problem in $A$ and $B$.

\bigskip

 \begin{lem} \label{conj}

Let $A$ and $B$ be non-trivial groups where the action of $B$ on a set $S \neq \emptyset$ is transitive.  If the conjugacy problem is solvable in $A \wr_S B$ then the conjugacy problem is solvable in $A$ and in $B$.

\end{lem}

\begin{proof}

We first note that by Lemma \ref{coset} for any $\beta \in B$ we can determine to which $H$-coset $\beta$ belongs.   

We recall that if $x=bf$ and $y=cg$ are conjugate in $A \wr_S B$ then there exists a conjugator $z =dh$, $d \in B$, $h \in A^{(S)}$ such that $(db, h^{b}f) = (cd, g^{d}h)$. 

To solve the conjugacy problem in $B$, we define:

\[ B \xrightarrow{\rho} A\wr_S B \xrightarrow{\psi} B\]

where $\rho(b) = (b,1)$ and $\psi(bf) = b$ so $\psi \circ \rho (b)= b$ for all $b \in B$.

We note that if $b\sim_B c$ then $\rho(b) \sim _{A \wr_S B} \rho(c)$.  As we can solve the conjugacy problem in $A \wr_S B$ we can determine if $\rho(b) \sim _{A \wr_S B} \rho(c)$.  If $\rho(b) \sim _{A \wr_S B} \rho(c)$ there exists $dh \in A\wr_S B$ such that $(db, h^b) = (cd, h)$ so $b \sim_B c$.  Hence we can solve the conjugacy problem in $B$.

To solve the conjugacy problem in $A$, we show that $a_1$ and $a_2$ are conjugate in $A$ if and only $x = (1_B, f)$ is conjugate to $ y = (1_B,g)$ in $A \wr_S B$, where $f = a_1^{(s_0)}$ and $g = a_2^{(s_0)}$.  

If $x$ and $y$ are conjugate in $A \wr_S B$ then there exists a $z =dh$ such that for all $s \in S$:

\[h(s)f(s) = g(s d^{-1})h(s) \]

If $s \neq s_0$ then $f(s) = 1_A$, $h(s) = g(s d^{-1})h(s)$ and $g (s d^{-1}) = 1_A$ which requires $s d^{-1} \neq s_0$ for all $s \neq s_0$.   Therefore $s_0 d^{-1} = s_0$ and $g(s_0d^{-1}) = a_2$.

As $g(s_0d^{-1}) = a_2$,  $h(s_0)a_1 = a_2 h(s_0)$ and so $a_1$ and $a_2$ are conjugate in $A$.

Conversely, if $x$ and $y$ are not conjugate in $A \wr_S B$, then (as $d1_B = 1_Bd$ for all $ d \in B$), there can exist no $dh \in A \wr_S B$ such that for all $s \in S$:

\[h(s)f(s) = g(s d^{-1})h(s) \]  

As seen from above, if $s_0 d^{-1} \neq s_0$ then this can never be satisfied, as then $g(s_0d^{-1}) = 1_A$ but $f(s_0) = a_1 \neq 1_A$ and so:

\[h(s_0)f(s_0) = h(s_0) a_1 \neq h(s_0) = g(s_0 d^{-1})h(s_0) \]

so the two sides will never be equal.  

Lastly, we need to check where $s_0 d^{-1} = s_0$.  This requires $h(s_0)a_1 \neq a_2 h(s_0)$ whatever the value of $h(s_0)$ and so $a_1$ and $a_2$ are not conjugate in $A$.

Therefore if we can solve the conjugacy problem in $A \wr_S B$ then we can solve it in $A$ and in $B$.

\end{proof}

Lastly we consider the Solution Sets.  

\bigskip

\begin{lem} \label{necSSab}

If $A$ is a group with an infinite number of conjugacy classes and we can solve the conjugacy problem in $A \wr_S B$, for any $\beta \in B$ and any finite set of pairs $(\alpha_i, \gamma_i)$ $\alpha_i, \gamma_i \in B$ for $1 \leq i \leq n$ we can determine whether:

\[\big{[}\bigcap_{i=1}^n \alpha^{-1}_iH\gamma_i \langle \beta \rangle \big{]}  \cap C_B(\beta) = \emptyset \]

\end{lem}
\begin{proof}

As double cosets of a group partition the group, either $H  \alpha_i \langle \beta \rangle = H \alpha_j \langle \beta \rangle$ or ${\alpha_j \notin H \alpha_i \langle \beta \rangle}$, which we can determine, because by Lemma \ref{coset} we can solve the membership problem in $H \alpha_i \langle \beta \rangle$.  We can therefore determine whether any two $\alpha_i, \alpha_j$ are in the same $\beta$-orbit.  

If $\alpha_i$ and $\alpha_j$ are in the same $\beta$-orbit then as the elements of $C_B(\beta)$ permute the $\beta$ orbits, the Solution Set is empty unless $\gamma_i \in H \gamma_j \langle \beta \rangle$, which again we can determine.  We may therefore assume without loss of generality that, for all $i \neq j$ $\alpha_i$ and $\alpha_j$ are in different $\beta$-orbits and $\gamma_i$ and $\gamma_j$ are in different $\beta$-orbits.  We therefore assume that the sets $\{H\alpha_i \langle \beta \rangle\}$ and $ \{H \gamma_i \langle \beta \rangle\}$ each contain $n$ different $\beta$-orbits in $S$.

We take a set $\{a_i\}$ of $n$ non-identity elements of $A$ such that $a_i \nsim_A a_j$ if $i \neq j$ and let $f = \sum_{i=1}^n  (a_i)^{(s_0 \alpha_i)}$ and $g = \sum_{i=1}^n  (a_i)^{(s_0 \gamma_i)}$.  If $\beta f \sim_W \beta g$ then from above there must exist a $d\in C_B(\beta)$ that effects a compliant permutation such that, for every $\alpha_i$, $s_0\alpha_id= s_0\gamma_i\beta^j$ for some $j \in \Z$.  Therefore, as we can determine if $\beta f \sim_W \beta g$, we can determine if the Solution Set is empty. 

\end{proof}

\bigskip

\begin{prop}

Let $A$ and $B$ are non-trivial recursively presented groups, where $A$ has an infinite number of conjugacy classes and $B$ acts on a set $S$ transitively, so $S = B/H$ where $H$ is the stabiliser of some arbitrary point $s_0 \in S$.   If the conjugacy problem in the permutational restricted wreath product $W = A \wr_S B$ is solvable then:

\begin{enumerate}[(i)]
\item the conjugacy problem is solvable in $A$ and in $B$;
\item either $A$ is abelian or the orbit order problem is solvable in $B$;
\item the membership problem for $H \gamma \langle \beta \rangle$ (for any $\gamma, \beta \in B$) is solvable; and
\item for any $\beta \in B$ and any finite set of $n$ pairs of elements $(\alpha_i, \gamma_i)$, where $\alpha_i, \gamma_i \in B$ and $1 \leq i \leq n$, we can determine whether or not:
\[\big{[}\bigcap_{i=1}^n \alpha^{-1}_iH\gamma_i \langle \beta \rangle \big{]}  \cap C_B(\beta) = \emptyset \]

\end{enumerate}

\end{prop} 

\begin{proof}

We have proved this as follows:

\begin{enumerate} [(i)]
\item Lemma \ref{conj};
\item Lemma \ref{orderproof};
\item Lemma \ref{coset}; 
\item Lemma \ref{necSSab}.
 \end{enumerate}  

 \end{proof}

This completes the proof of Theorem \ref{conditions}.

\subsection {Assumption that A has an Infinite Number of Conjugacy Classes}

We return to our assumption in Theorem \ref{conditions} that $A$ has an infinite number of conjugacy classes. This assumption is only used in Lemma \ref{necSSab}, and so only to prove that it is necessary to be able to determine whether the Solution Set is empty.  

For example, in \cite{O1} Osin identifies a finitely generated infinite group that has only two conjugacy classes.  If $A$ were such a group, then applying Lemma \ref{n=m} tells us that if $bf \sim_W bg$ then $f$ and $g$ contain the same number of non-trivial $b$ orbits. If they do, then all permutations of the non-trivial orbits $\{O_i\}$ will be compliant permutations.  It is therefore not necessary to determine to which $b$ orbit $d$ maps $O_i$, but only that the exists an element of $C_B(b)$ that maps each $O_i$ to another non-trivial orbit.  In other words, rather than $d$ needing to be an element of $\big{[}\bigcap_{i=1}^n \alpha^{-1}_iH\gamma_i \langle \beta \rangle \big{]}$, we only need to determine that, for each $O_i$:

\[d \in \big{[}\bigcup_{j=1}^n \alpha^{-1}_iH\gamma_j \langle \beta \rangle \big{]}   \]

\subsection{Two Remarks}

We finish this section with two remarks.

\bigskip

\begin{rem} \label{defnh}

In the course of the proof of Theorem \ref{conditions} if the action of $B$ on $S$ is transitive and $bf \sim_W cg $ we have found a conjugator $dh \in W$ such that ${(dh)(bf) = (cg)(dh)}$.

First, we note that $d$ is an element of the Solution Set that satisfies $db=cd$.  In addition, from Lemmas \ref{infin2} and \ref{fin2}:

\begin{enumerate} [(i)]

 \item if $A$ is abelian and $|O_i| =n \leq N$ then:

\[ h(s_0t_i b^k) = \big{[} \displaystyle \prod_{j =0}^k g(s_0t_i b^j d^{-1}) \big{]}^{-1} \big{[} \displaystyle \prod_{j =0}^k f(s_0t_i b^j) \big{]} \]

    \item if $A$ is abelian and $|O_i| > N$ then:

\[ h(s_0t_i b^k) = \big{[} \displaystyle \prod_{j =M}^k g(s_0t_i b^j d^{-1}) \big{]}^{-1} \big{[} \displaystyle \prod_{j =M}^k f(s_0t_i b^j) \big{]} \]

\item if $A$ is not abelian and $|O_i| = n$ finite then $g_i^{*d} = \gamma_i f^*_i \gamma_i^{-1}$ for some $\gamma_i \in A$ and:

\[ h(s_0t_i b^k) = \big{[} \displaystyle \prod_{j =0}^k g(s_0t_i b^j d^{-1}) \big{]}^{-1} \gamma_i \big{[} \displaystyle \prod_{j =0}^k f(s_0t_i b^j) \big{]} \]

\item if $|O_i|$ is infinite:

\[ h(s_0t_i b^k) = \big{[} \displaystyle \prod_{j =M}^k g(s_0t_i b^j d^{-1}) \big{]}^{-1} \big{[} \displaystyle \prod_{j =M}^k f(s_0t_i b^j) \big{]} \]
\end{enumerate}

\end{rem} 

\bigskip

 \begin{rem} \label{SS}

 If $H \leq B$ and $\alpha_i \in T$ then let $H_i = \alpha_i^{-1}H\alpha_i$.
 
\bigskip

 \end{rem} 
 We can then write:

 \[  \big{[}  \bigcap_{i=1}^n \alpha_i^{-1}H\gamma_i \langle b \rangle \big{]} \cap C_B(b) \]

 as:
 
 \[ \big{[} \bigcap_{i=1}^n H_i\alpha_i^{-1}\gamma_i \langle b \rangle \big{]} \cap C_B(b)\]

 We note that $H_i$ is the stabiliser of $s_0\alpha_i$ as:

\[(s_0\alpha_i)(\alpha_i^{-1}H\alpha_i) = s_0H\alpha_i = s_0\alpha_i\]

Next, $\alpha_i^{-1}h\gamma_ib^i \in C_B(\beta)$ for some $h \in H$ and $j \in \Z$ if and only if $\alpha_i^{-1}h\gamma_i \in C_B(\beta)$.  Let $x_i = \alpha_i^{-1}\gamma_i$, and $H_i = \gamma_i^{-1} H \gamma_i$, so $H_i \leq B$.  We can then rewrite the Solution Set as:

\begin{align*}
\big{[}  \bigcap_{i=1}^n \alpha_i^{-1}H\gamma_i \langle \beta \rangle \big{]} \cap C_B(\beta) &=   \bigcap_{i=1}^n \big{[} x_i H_i \cap C_B(\beta)\big{]} \langle \beta \rangle \\
&= \bigcap_{i=1}^n \langle \beta \rangle \big{[} x_i H_i \cap C_B(\beta)\big{]}  
\end{align*}
 \section{Some First Applications} \label{applic}

 Matthews's proof of Theorem \ref{B} (subject to the discussion in Section \ref{OO} above) is set out in \cite{M1}.   Our first application of Theorem \ref{conditions} is to provide an alternative proof of Theorem \ref{B}.

\textbf{Theorem 1.0.1.} \textit{The standard restricted wreath product $A \wr B$ of two nontrivial groups $A$ and $B$ both of Class $\mathcal{C}$ is itself of Class $\mathcal{C}$ if and only if (i) $B$ has a solvable power problem and (ii) either $A$ is abelian or $B$ has a solvable order problem.}

\begin{proof}

We note that $A \wr B = A\wr_S B$ where $S=B$ so $H$ is trivial. We can therefore simplify the conditions in Theorem \ref{conditions} because: 

\begin{enumerate} [(i)]
    \item $B$ acts transitively on itself;
    \item the $b$-orbits are the $\langle b \rangle$ cosets in $B$ so the orbit order problem is the order problem in $B$; and
    \item the membership problem for $H \gamma \langle \beta \rangle$ becomes the power problem in $B$.    \end{enumerate}

Lastly in $A \wr B$ we can solve the the Solution Set if we can determine if there exists a $d\in B$ such that:

\[d \in \big{[}\bigcap_{i=1}^n \alpha^{-1}_i\gamma_i \langle \beta \rangle \big{]}  \cap C_B(\beta)\] 

If $n = 1$ then $d$ exists if and only if $\alpha^{-1}\gamma \in C_B(b)$ which is the word problem in $B$ and so solvable.

If $n >1$ if $d$ exists then:

\begin{enumerate} [(i)]
\item for every $1 \leq i \leq n$, $\alpha_i^{-1}\gamma_i \in C_B(b)$; and    \item for every pair $(i_1, i_2)$, $i_1 \neq i_2$, $d = \alpha_{i_1}^{-1}\gamma_{i_1}b^{j_1} = \alpha_{i_2}^{-1}\gamma_{i_2}b^{j_2}$ and so $\gamma_{i_1}^{-1}\alpha_{i_1}\alpha_{i_2}^{-1}\gamma_{i_2} \in \langle \beta \rangle$.
     \end{enumerate}

(i) above is the word problem in $B$ and (ii) is the power problem in $B$, so both are solvable.  Therefore in $A \wr B$ it is possible to determine whether or not the Solution Set for a compliant permutation is empty.

\end{proof}

The proof of Theorem \ref{conditions} also provides some initial checks that might enable us to conclude that $x=bf$ and $y=cg $ are not conjugate in $A \wr_SB$. For example, $bf \nsim_W cg$ if:

\begin{enumerate}[(i)]
\item $b \nsim_B c$; 
\item $\bar{f} = 1$ but $\bar{g} \neq 1$;
\item $|\sigma(\bar{f})| \neq |\sigma(\bar{g})|$; or
\item there is no possible compliant permutation of the $b$ orbits in $\bar{f}$ and $\bar{g}$.
\end{enumerate}

The following is a direct application of Theorem \ref{conditions}

\textbf{Corollary 1.0.6.} \textit{Let $A$ and $B$ are non-trivial recursively presented groups, where $A$ has an infinite number of conjugacy classes and $B$ is finite and acts on a set $S$ transitively, so $S = B/H$ where $H$ is the stabiliser of some arbitrary point $s_0 \in S$.  The conjugacy problem in the permutational restricted wreath product $A \wr_S B$ is solvable if and only if the conjugacy problem is solvable in $A$.}

\begin{proof}

We note that if $B$ is finite then the conjugacy problem, the orbit order problem and the membership problem for $H \gamma \rangle \beta \langle$ are all solvable in $B$ and it is also possible to test whether any element of $B$ is in the Solution Set.
    
\end{proof}
Similarly, if $B$ is infinite group but we are able to identify all the elements in $C_B(b)$, it may be possible to take each possible element of $ C_B(b)$ in turn and test whether it is an element of the Solution Set.

 \section{Theorem 1.0.4 - Transitivity assumption}

We now return to our assumption in Theorems \ref{oneway} and \ref{conditions} that the action of $B$ on $S$ is transitive.  

If the action of $B$ on $S$ is not transitive, because $S$ contains more than one $B$-orbit, we may partition $S$ into these $B$-orbits $S_k$ such that:

\[ S = \sum_k S_k = \sum_k B/H_k\]

where for each $S_k$, $B$ acts on $S_k=B/H_k$ and $H_k$ is the stabiliser of some arbitrary point $s_k$ in $S_k$. We may therefore partition any $f \in A^{(S)}$ such that $f = \sum_k f_k$ where each $f_k \in A^{(S_k)}$.

\bigskip

\begin{lem} \label{gamma}

Let $H \leq G$ and $\psi:G \twoheadrightarrow H$ be a surjective homomorphism.  If $x \sim_G y$ then $\psi(x) \sim_H \psi(y)$.
    
\end{lem}

\begin{proof}

If $x \sim_G y$ there exists $z \in G$ such that $zx = yz$.  Then:

\begin{align*}
\psi(zx) &= \psi(yz) \\  
\Rightarrow \psi(z)\psi(x) &= \psi(y)\psi(z) \\
\Rightarrow \psi(x) &\sim_H \psi(y)
\end{align*}.

\end{proof}

\begin{lem} \label{trans3}

Let $A$ and $B$ be non-trivial recursively presented groups, where $B$ acts on a set $S$ such that there are $m$ finite $B$-orbits in $S$, so:

\[S =\sum_{k=1}^m  S_k =  \sum_{k=1}^m  B/H_k \]

where each $S_k$ is a $B$-orbit in $S$ and each $H_k$ is the stabiliser of some arbitrary point in $S_k$.  Let $bf, bg \in W= A \wr_S B$ where $b \in B$, $f_k, g_k \in A^{(S_k)}$ and $f = \sum_k f_k$ and $g = \sum_k g_k$.  ${(dh_k)( bf_k) = (bg_k)( d h_k)}$ for all $k$ if and only if $(dh)(bf) = (bg)(dh)$ where $d \in B$ and ${h = \sum_k h_k \in A^{(S)}}$.

\end{lem}

\begin{proof}

We first note that if ${(dh_k)( bf_k) = (bg_k)( d h_k)}$ then $d \in C_B(b)$ and $d$ effects a compliant permutation in each $S_k$ and therefore in $S$.  In addition, from Remark \ref{defnh} we can see that ${h = \sum_k h_k}$ is such that $dh$ is a conjugator for each $B$ orbit and therefore in $A \wr_S B$.
    
\end{proof}

\textbf{Theorem 1.0.4.} \textit{Let $A$ and $B$ be non-trivial recursively presented groups, where $B$ acts on a set $S$ such that there are a finite number $m$ of $B$-orbits in $S$, so:}

\[S =\sum_{k=1}^m  S_k =  \sum_{k=1}^m  B/H_k \]

\textit{where each $S_k$ is a $B$-orbit in $S$ and each $H_k$ is the stabiliser of some arbitrary point in $S_k$.  The conjugacy problem in the permutational restricted wreath product $A \wr_S B$ is solvable if:}

\begin{enumerate}[(i)]
\item \textit{the conjugacy problem is solvable in $A$ and in $B$;
\item either $A$ is abelian or the orbit order problem is solvable in $B$;}
\item \textit{for each $H_k$ the membership problem for $H_k \gamma \langle \beta \rangle$ (for any $\gamma, \beta \in B$) is solvable;}
\item \textit{if for each $S_k$, for any $\beta \in B$ and any finite set of $n$ pairs of elements $(\alpha_i, \gamma_i)$, where $\alpha_i, \gamma_i \in B$ and $1 \leq i \leq n$: }

\[\Gamma_k \coloneq  \big{[}\bigcap_{i=1}^n \alpha^{-1}_iH_k\gamma_i \langle \beta \rangle \big{]}  \cap C_B(\beta) \]

\textit{we can determine whether $ \bigcap_{k=1}^m \Gamma_{k} =\emptyset$.} 
\end{enumerate} 
  
\begin{proof}

Applying Theorem \ref{oneway}, we can solve the conjugacy problem in $W_k = A \wr_{S_k} B$.  We take arbitrary $bf, cg \in W = A\wr_S B$.  If $bf \sim_W cg$ then $b \sim_B c$ so if $b \nsim_B c$ then we are done and otherwise as before we may assume $b=c$. 

We note that for each $S_k$ there exists a projection $\psi_k: W\twoheadrightarrow W_k$ and by assumption we can determine if $\psi_k(bf) \sim_{W_k} \psi_k(cg)$. By Lemma \ref{gamma} above if  $\psi_k(bf) \nsim_{W_k} \psi_k(cg)$ for some $1 \leq k \leq m$ then $bf \nsim_W cg$ and we are done.

If $\psi_k(bf) \sim_{W_k} \psi_k(cg)$ for all $k$, we first assume that there is only one compliant permutation in each $S_k$. Therefore for each $S_k$ there exists a conjugator $d_kh_k$ such that $d_k \in \Gamma_k$ for that compliant permutation.  Applying Lemma \ref{trans3}, if $bf \sim_W cg$ there must exist $d \in B$ such that $d \in \bigcap_{k=1}^m \Gamma_k$. 

If there is more than one compliant permutation in one or more $S_k$, there can be only a finite number of them and so only a finite number of $\bigcap_{k=1}^m \Gamma_k$ to check.

\end{proof}

\section{Abelian and Polycyclic Groups} \label{poly}

A group $G$ is polycyclic if there exist $G_i$ and $n<\infty$ such that:

\[ G = G_1 \trianglerighteq G_2 \ldots G_{n} \trianglerighteq G_{n+1} = \{1\} \]

and the quotient group $G_{i} / G_{i+1}$ is a cyclic group for all $1 \leq i < n$. A polycyclic group is therefore finitely presented and residually finite (Theorem 1, Chapter 1 \cite{S2}).  If $G$ is polycyclic, any subgroup $H \leq G$ is also polycyclic.  In addition, if $N$ is any normal subgroup of a polycyclic group $G$, then $G/N$ is polycyclic. 

A polycyclic group is residually finite (Theorem 1 Chapter 2 \cite{S2}) and conjugacy separable (see \cite{F2}) and so the word and conjugacy problems are solvable.  In addition, algorithms exist to determine the order of any element of a polycyclic group (see, for example, \cite{S3}).

Any subgroup of a polycyclic group $G$ is closed in the profinite topology on $G$ (Mal'cev \cite{M4}). In addition, in \cite{L1}, Lennox and Wilson show that every double coset in a polycyclic group $G$ is closed in the profinite topology on $G$.

\bigskip

\begin{lem} \label{AM}

For any $H,K \leq G$ and any $x,y \in G$:

\begin{enumerate}[(i)]
    \item $xH \cap yK \neq \emptyset$ if and only if $xy^{-1} \in HK$; 
\item if $G$ is polycyclic we can determine if $xy^{-1} \in HK$,  and
    \item $xH \cap yK \neq \emptyset$ if and only if $xH \cap yK = w(H \cap K) $ for some $w \in xH \cap yK$.
    \end{enumerate} 

\end{lem} 

\begin{proof}

There exists $z \in  xH \cap yK$ if and only if $z = xh = yk$ for some $h \in H$ and $k \in K$ and so $x^{-1}y \in HK$.  

As $HK$ is closed in the profinite topology on $G$, we can solve the membership problem in $HK$ and therefore determine whether $xy^{-1} \in HK$ and therefore if $xH \cap yK = \emptyset$.

If $xH \cap yK \neq \emptyset$, let $w \in xH \cap yK$.  Then $w(H \cap K) \in xH$ and $w(H \cap K) \in yK$ so:

\[w(H \cap K) \subseteq xH \cap yK\]

Similarly, if $z \in xH$ then $w^{-1}z \in xH$ and if $z \in yK$ then $w^{-1}z \in yK$.  Therefore if $z \in xH \cap yK$, $z \in w(H \cap K)$.

Therefore $xH \cap yK = w(H \cap K) $.

\end{proof}

There are known algorithms to compute a finite generating set for $H \cap K$ (see, for example, Section 8.4.3 \cite{E1}).

\bigskip

\textbf{Corollary 1.0.7.} \textit{If $A$ is a recursively presented group in which the conjugacy problem is solvable and $B$ is a polycyclic group that acts on a set $S$ transitively, then the conjugacy problem is solvable in $A \wr_S B$.}

\begin{proof}

Applying Theorem \ref{conditions}, we need to show:

\begin{enumerate}[(i)]
\item the conjugacy problem is solvable in $B$;
\item the orbit order problem is solvable in $B$;
\item the membership problem for $H \gamma \langle \beta \rangle$ (for any $\gamma, \beta \in B$) is solvable; and
\item we can determine whether or not the Solution Set for any compliant permutation is empty.
   \end{enumerate}  
  
Taking each in turn:

(i) As $B$ is finitely presented and conjugacy separable, the conjugacy problem is solvable in $B$ (Mostowski \cite {M8}).

(ii) As discussed above, we can determine the order of $b \in B$.  If $|b|$ is finite then $|Hg\langle b \rangle|$ is finite.  If $|b|$ is infinite, then by Lemma \ref{OOint}, $|Hg\langle b \rangle|$ is finite if and only if:

\[H \cap \langle g b g^{-1} \rangle \neq \{1_B \}\]

As discussed above, there are known algorithms to compute a finite generating set for the intersection $H \cap \langle g b g^{-1} \rangle$ and a finite number of applications of the word problem will determine whether the intersection is trivial.   Hence the orbit order problem is solvable.

(iii) Any subgroup or double coset of $ B$ is closed in the profinite topology on $G$.  Therefore, as $\langle \gamma \beta \gamma^{-1} \rangle \leq B$ and right multiplication is continuous, $H\gamma \langle \beta \rangle$ is also closed.  We can therefore solve the membership problem for $H \gamma \langle \beta \rangle$ (for any $\gamma, \beta \in B$) .

(iv)  From Remark \ref{SS}, we can then rewrite the Solution Set as:

\[\big{[}  \bigcap_{i=1}^n \alpha_i^{-1}H\gamma_i \langle \beta \rangle \big{]} \cap C_B(\beta) =   \bigcap_{i=1}^n \big{[} x_i H_i \cap C_B(\beta)\big{]} \langle \beta \rangle  \]

Starting with $i=1$ by Lemma \ref{AM} we can determine if $x_1 H_1 \cap C_B(\beta) = \emptyset$.  If it is, then the Solution Set is empty.  If it is not, then applying Lemma \ref{AM}, $x_1H_1 \cap C_B(\beta) = w_1K_1$ for some $w_1 \in x_1H_1 \cap C_B(\beta)$, where $K_1 = H_1 \cap C_B(\beta)$.  $K_1 \leq C_B(\beta)$, so $K_1$ centralises $\langle \beta \rangle$, and $L_1 = K_1\langle \beta \rangle $ is a subgroup of $B$ (because, if $x,y \in K_1$ then $x\beta ^jy^{-1}\beta^k \in L_1$). 

We can therefore rewrite the Solution Set as:

\[ w_1L_1 \cap   \bigcap_{i=2}^n \big{[} x_i H_i \cap C_B(\beta)\big{]} \langle \beta \rangle  \]

noting that (as above) we can find a finite generating set for $L_1$. 

We now apply this same method to $i=2$, applying Lemma \ref{AM} to determine if:

\[x_2 H_2 \cap C_B(\beta) = \emptyset\]

(where, if it is, the Solution Set is empty) or, if not, finding a finite generating set for $L_2 = K_2\langle \beta \rangle $ and for $w_1L_1 \cap w_2L_2$.  If $w_1L_1 \cap w_2L_2 = \emptyset$ then the Solution Set is empty, if it is not, we continue to $i=3$ and so on, until either we determine that, for some $i=j$:

\[x_j H_j \cap C_B(\beta) = \emptyset \]

or that:

\[\bigcap_{i=1}^j w_iL_i = \emptyset \]

(so the Solution Set is empty) or we have rewritten the Solution Set as:

\[ \bigcap_{i=1}^n   w_iL_i  \]

where each $L_i$ is a subgroup of $B$.  

This is the intersection of a finite number of cosets and so, by Lemma \ref{AM} we can determine whether the Solution Set is empty.

\end{proof}

As all finitely generated abelian groups are polycyclic, we have the following Corollary:

\textbf{Corollary 1.0.8. }\textit{If $A$ is a recursively presented group in which the conjugacy problem is solvable and $B$ is a finitely generated abelian group that acts on a set $S$ transitively, then the conjugacy problem is solvable in $A \wr_S B$.}

\section{B Hyperbolic and H Quasi-convex} \label{hyp}

A hyberbolic group is finitely presented (see \cite{B3} Section \textit{III}$ \: \Gamma \: 2.2$) and has a finite Dehn presentation (see \cite{B3} Section \textit{III} $\: \Gamma \: 2.6$). 

The centraliser of any element of a hyperbolic group and any cyclic sub-group of a hyperbolic group are both quasi-convex (\cite{B3} \textit{III} $  \: \Gamma \: 3.9$).  Therefore if $\beta \in B$ and $B$ is hyperbolic, $\langle \beta \rangle$ and $C_B(\beta)$ are both quasi-convex and therefore hyperbolic (\cite{B3} \textit{III} $ \: \Gamma \: 3.7$).  The intersection of two quasi-convex subgroups of a hyperbolic group is quasi-convex (\cite{B3} \textit{III}  $\: \Gamma \: 3.9$).

If $B$ is a finitely generated free group, we can construct Stallings graphs to solve various membership problems, including the membership problem for $H \gamma K$ where $H,K \leq B$ and $\gamma \in B$.  This technique can be extended to hyperbolic groups, so long as $H$ and $K$ are both quasi-convex.  For further discussion of this see Kharlampovich, Miasnikov and Weil \cite{K3}, where they explain how to compute a finite Stallings graph of a quasi-convex subgroup of a hyperbolic group and use this to solve a number of algorithmic problems.  The following Lemma is a collection of their results:

\bigskip

\begin{lem} Propositions $6.2$, $6.3$ and $6.4$ \cite{K3} \label{DCHyp}

If $H,K \leq G$ where $G$ is a hyperbolic group and $H$ and $K$ are quasi-convex, then:

\begin{enumerate}[(i)]
\item for any $x \in G$ we can determine if $\langle x \rangle$ is finite;
\item for any $g,x \in G$ we can produce an automaton generating all the elements of $Hx$, $xH$ and $HxK$ and therefore decide whether or not $g \in Hx$, $g \in xH$ and $g \in HxK$; 
\item for any $Hx$ and $Ky$ it is possible to determine if $Hx \cap Ky = \emptyset$;
\item it is possible to compute a generating set for $H \cap K$. 
\end{enumerate}
\end{lem}

We can therefore prove the following:

\bigskip

\textbf{Corollary 1.0.9.} \textit{If $A$ is a recursively presented group in which the conjugacy problem is solvable and $B$ is a hyperbolic group that acts on a set $S$ transitively, with $S = B/H$ and $H \leq B$ is quasi-convex then the conjugacy problem is solvable in $A \wr_S B$.}

\begin{proof}

Applying Theorem \ref{conditions}, we need to show that is $B$ is hyperbolic and $H$ is quasi convex then:

\begin{enumerate}[(i)]
\item the conjugacy problem is solvable in $B$;
\item the orbit order problem is solvable in $B$;
\item the membership problem for $H \gamma \langle \beta \rangle$ (for any $\gamma, \beta \in B$) is solvable; and
\item we can determine whether or not the Solution Set for any compliant permutation is empty.
   \end{enumerate}  

Taking each in turn:

(i) $B$ has a solvable conjugacy problem (see \cite{B3} Section \textit{III}$\: \Gamma \: 2.8$);

(ii) By Lemma \ref{DCHyp} we can determine if $|\beta|$ is finite.  If $|\beta|$ is finite then $|Hg\langle \beta \rangle|$ is finite.  If $|\beta|$ is infinite, then by Lemma \ref{OOint} $|Hg\langle \beta \rangle|$ is finite if and only if $H \cap \langle g \beta g^{-1} \rangle \neq \{1_B \}$ which from Lemma \ref{DCHyp} is solvable.
  
(iii) From Lemma \ref{DCHyp}, the membership problem for $H \gamma \langle \beta \rangle$ (for any $\gamma, \beta \in B$) is solvable.

(iv) We divide this into two cases.    

\textit{Case 1: $|\beta|$ infinite}

If $\beta \in B$ has infinite order, $C_B(\beta)$ is virtually cyclic with the subgroup $ \langle \beta \rangle$ having finite index in $C_B(\beta)$ (\cite{B3} \textit{III}  $\: \Gamma \; 3.10$) .  Therefore $C_B(\beta) = \bigcup_l \rho_l \langle \beta \rangle$ where $l$ is finite and we can find the finite set of coset representatives $\{\rho_l\} \subseteq B$. An algorithm exists to compute $\{\rho_l\}$ (Proposition 4.11 \cite{B1}).

Applying this to our Solution Set, if:

\[ \ \big{[} \bigcap_{i=1}^n \alpha_i^{-1}H \gamma_{i} <\beta> \big{]} \cap C_B(\beta) \neq \emptyset\] 

then, for some $\rho_l$, for all $i$:

\[  \rho_l \in \alpha_i^{-1} H \gamma_i \langle \beta \rangle \]

This is the double coset membership problem for a finite number of $\rho_l$.  Hence, from (iii) above, we can determine if the Solution Set is empty.

\textit{Case 2: $|\beta|$ finite}

If $|\beta| = p$ finite, then: 

\[\alpha_i^{-1}H \gamma_{i} \langle \beta \rangle = \bigcup_{j_i=1}^p \alpha_i^{-1}H \gamma_{i} \beta^{j_i} \]

As $p$ is finite we can take a finite number of tuples representing one choice of $j_i$ for each $i$, and apply Lemma \ref{DCHyp} to determine whether or not:

\[\bigcap_{i=1}^n \alpha_i^{-1}H \gamma_{i} \beta^{j_i} = \emptyset \]

We can discard any such tuple where this is the case, so leaving us with a finite number of tuples each representing a set $\{j_i\}$ where 

\[\bigcap_{i=1}^n \alpha_i^{-1}H \gamma_{i} \beta^{j_i} \neq \emptyset\]

Each such tuple represents the non-empty intersection of a finite number of cosets of $H$ and (applying Lemma \ref{DCHyp}) for each we can test whether its intersection with $C_B(\beta)$ is empty. 

Hence, if $\beta$ is finite we can determine whether the Solution Set is empty.

\end{proof}

\end{document}